\documentclass[12pt,a4paper,reqno]{amsart}
\usepackage{color}
\usepackage{amsfonts,amsmath,amssymb,amsxtra,url,float} 
\allowdisplaybreaks[4] 
\usepackage[colorlinks,linkcolor=RoyalBlue,anchorcolor=Periwinkle,citecolor=Orange,urlcolor=Green]{hyperref} 
\usepackage[usenames,dvipsnames]{xcolor} 
\usepackage{enumitem}
\usepackage{geometry} 
\usepackage{rotating} 
\usepackage{lscape} 
\DeclareSymbolFont{largesymbol}{OMX}{yhex}{m}{n}
\DeclareMathAccent{\Widehat}{\mathord}{largesymbol}{"62} 
\usepackage{multirow}
\usepackage{graphicx} 
\usepackage{subfigure} 
\usepackage{tikz}
\usetikzlibrary{decorations.pathreplacing,decorations.markings}
 \tikzset{
  on each segment/.style={
    decorate,
    decoration={
      show path construction,
      moveto code={},
      lineto code={
        \path [#1]
        (\tikzinputsegmentfirst) -- (\tikzinputsegmentlast);
      },
      curveto code={
        \path [#1] (\tikzinputsegmentfirst)
        .. controls
        (\tikzinputsegmentsupporta) and (\tikzinputsegmentsupportb)
        ..
        (\tikzinputsegmentlast);
      },
      closepath code={
        \path [#1]
        (\tikzinputsegmentfirst) -- (\tikzinputsegmentlast);
      },
    },
  },
  mid arrow/.style={postaction={decorate,decoration={
        markings,
        mark=at position 0.6 with {\arrow[#1]{stealth}} 
      }}},
}
\usetikzlibrary{arrows}
\usetikzlibrary{trees}

\usetikzlibrary{matrix}
\usetikzlibrary{patterns}
\usetikzlibrary{shadings} 
\usepackage[all]{xy} 
\usepackage{mathdots} 
\usepackage{dsfont} 
\usepackage{cite}
\usepackage{mathrsfs} 
\usepackage{multicol} 

\usepackage{bbm}           
\usepackage{bm}
\usepackage{changepage}

\numberwithin{figure}{section}
\usepackage{marginnote} 
\usepackage{graphicx} 
\usepackage{multicol} 

\usepackage{fancyhdr} 
\newtheorem{theorem}{Theorem}[section]
\newtheorem{lemma}[theorem]{Lemma}

\newtheorem{main theorem}[theorem]{Main Theorem}
\newtheorem{proposition}[theorem]{Proposition}
\newtheorem{definition}[theorem]{Definition}
\newtheorem{construction}[theorem]{Construction}
\newtheorem{remark}[theorem]{Remark}
\newtheorem{example}[theorem]{Example}

\usetikzlibrary{arrows}

\numberwithin{equation}{section}

\definecolor{Green}{RGB}{0 139 0}

\begin{document}

\title[On BD-algebra and CM-Auslander algebra for a gentle algebra]
      {On BD-algebra and CM-Auslander algebra for a gentle algebra and their representation types}

\author{Mengdie Zhang}
\address{Department of Mathematics, Nanjing University, Nanjing 210093, China}
\email{602023210021@smail.nju.edu.cn / zmy1net@163.com}

\author{Yu-Zhe Liu$^*$}
\address{School of Mathematics and Statistics, Guizhou University, Guiyang 550025, Guizhou, P. R. China}
\email{liuyz@gzu.edu.cn / yzliu3@163.com (Y.-Z. Liu)}

\author{Chao Zhang}
\address{State Key Laboratory of Public Big Data, School of Mathematics and Statistics, Guizhou University, Guiyang, 550025, China}
\email{zhangc@amss.ac.cn}

\maketitle


\vspace{-3mm}


\begin{abstract}
Let $A$ be a gentle algebra, and $B$ and $C$ be its BD-gentle algebra and CM-Auslander algebra, respectively. 
In this paper, we show that the representation-finiteness of $A$, $B$ and $C$ coincide 
and the representation-discreteness of $A$, $B$ and $C$ coincide. 
\end{abstract}

\begin{adjustwidth}{1.3cm}{1.3cm} 
{ \small

\noindent

\noindent

}
\end{adjustwidth}

\section{Introduction}

\newcommand{\To}[2]{\mathop{-\!\!\!-\!\!\!\longrightarrow}\limits^{#1}_{#2}}
\def\Pic{Figure\ } \def\defines{\it\color{blue}}
\def\NN{\mathbb{N}} \def\RR{\mathbb{R}} \def\QQ{\mathbb{Q}} \def\CC{\mathbb{C}}

\def\kk{\mathds{k}} \def\Q{\mathcal{Q}} \def\I{\mathcal{I}} \def\scrC{\mathscr{C}}
\def\pdim{\mathrm{proj.dim}} \def\idim{\mathrm{inj.dim}} \def\flatdim{\mathrm{fl.dim}} \def\gldim{\mathrm{gl.dim}}
\def\source{\mathfrak{s}} \def\target{\mathfrak{t}} \def\e{\varepsilon}
\def\modcat{\mathsf{mod}} \def\Dcat{\mathsf{D}} \def\Gproj{\mathsf{G}\text{-}\mathsf{proj}}
\def\Hom{\mathrm{Hom}} \def\End{\mathrm{End}} \def\Ext{\mathrm{Ext}}

\def\Path{\wp}
\def\=<{\leqslant} \def\>={\geqslant}
\def\bbA{\mathbb{A}} \def\pmbA{\pmb{A}} \def\gentleSet{O}

\def\ind{\mathrm{ind}} \def\op{\mathrm{op}}
\def\rad{\mathrm{rad}} \def\soc{\mathrm{soc}}
\def\cyc{\mathrm{cyc}} \def\ncyc{\mathrm{ncyc}}

\def\str{\omega} \def\band{\rho} \def\M{\mathds{M}}
\def\Str{\mathsf{Str}} \def\Band{\mathsf{Ban}}
\def\hstr{\tilde{\omega}} \def\hband{\tilde{\rho}}


\def\pointsize{0.8mm}
\def\pointsizeII{0.5mm}
\def\pointsizeIII{0.8mm}
\def\rotateangle{25}

\def\white{\color{white}}
\def\inc{\mathrm{inc}}

\def\bfS{\mathbf{S}} \def\Surf{\mathcal{S}} \def\Marked{\mathcal{M}} \def\preAS{\widehat{\mathcal{A}}} \def\AS{\Gamma}

\def\rbullet{{\color{red}\pmb{\circ}}}
\def\bbullet{{\color{blue}\pmb{\bullet}}}

\def\Poly{\mathsf{P}} \def\PolyII{\mathsf{Q}} \def\marked{{\textsc{d}}} \def\bcomp{{\textsc{b}}}

\def\redP{{\Poly}_{\rbullet}} \def\frakE{\mathfrak{E}}
\newcommand{\reda}[2]{\mathsf{v}_{\rbullet,#1}^{#2}}
\def\redanoind{\mathsf{v}_{\rbullet}}
\def\rMarked{\mathcal{M}_{\rbullet}} \def\rAS{\mathcal{A}_{\rbullet}} \def\prerAS{\widehat{\mathcal{A}}_{\rbullet}}
\def\rpunc{\mathfrak{P}_{\rbullet}}  \def\rEP{\mathrm{EP}_{\rbullet}}

\def\buleP{{\Poly}_{\bbullet}}
\newcommand{\bluea}[2]{\mathsf{v}_{\bbullet,#1}^{#2}}
\def\blueanoind{\mathsf{v}_{\bbullet}}
\def\bMarked{\mathcal{M}_{\bbullet}} \def\bAS{\mathcal{A}_{\bbullet}} \def\prebAS{\widehat{\mathcal{A}}_{\bbullet}}
\def\bpunc{\mathfrak{P}_{\bbullet}}  \def\bEP{\mathrm{EP}_{\bbullet}}

\def\av{\mathfrak{v}}

\def\tc{\tilde{c}} \def\foliation{\mathcal{F}}


\def\BD{\mathrm{BD}}


\def\CMA{\mathrm{CMA}}


Let $\kk$ be a field. Gentle algebras are important algebras in representation theory.
Many algebras are closely related to gentle algebras,
such as almost gentle algebras, biserial algebras, special biserial algebras,
string algebras, hereditary Nakayama algebras, etc.
Gentle algebras are introduced by Assem and Skowro\'{n}ski in \cite{AS1987} at first,
authors use it to study the iterated tilted algebras of type $\tilde{\mathbb{A}}$.
Subsequently, the gentle algebras received widespread attention.

Base on the works of Wald and Waschb\"{u}sch in \cite{WW1985}, Butler and Ringel introduced string and band (see Definition \ref{def:str band}) in \cite{BR1987} to describe the finitely generated module categories of string algebras,
this work provides a method to study all indecomposable modules and irreducible morphisms
in the finitely generated module categories of gentle algebras, since each gentle algebra is a string algebra.
In \cite{ALP2016}, the authors described the bounded derived category of
gentle algebra by using homotopy strings and homotopy bands (see Definition \ref{def:hstr hband}).
In recent years, the authors of \cite{OPS2018,APS2023,BCS2021,QZZ2022} found a method to describe
strings, bands, homotopy strings, and homotopy bands by using the marked surfaces, the geometric models given in \cite{HKK2017}, of gentle algebras, we obtained a geometric method to consider the properties of gentle algebras.
Thus, many questions about gentle algebras can be transformed into geometry and combination questions, such as \cite[etc]{CJS2022,CS2023b,Cpre2023,HZZ2023,FGLZ2023,LiuZhou2024}.

Let $A$ be a gentle algebra and $H$ be its normalization.
\begin{itemize}
  \item[(1)] The BD-gentle algebra of $A$ is defined as 
    \[B = \BD(A) := \left(\begin{matrix} A & H \\ \rad A & H \end{matrix}\right). \]
  \item[(2)] The CM-Auslander algebra of $A$ is defined as
    \[C = A^{\CMA} := \bigg(\End_A\bigg(\bigoplus\limits_{G \in \ind(\Gproj(A))} G\bigg)\bigg)^{\mathrm{op}}, \]
    where $\ind(\Gproj(A))$ is the set of all indecomposable Gorenstein-projective right $A$-modules (up to isomorphism).
\end{itemize}
In this paper, we study the (derived) representation types between $A$, $B$, and $C$,
and show the following results.

\begin{sloppypar}
\begin{theorem}[Proposition \ref{prop:repr type} and Theorem \ref{mainthm:repr-type}]
Let $A$ be a gentle algebra. Then the following statements are equivalent:
\begin{itemize}
  \item[\rm(1)] one of $A$, $B$ and $C$ is representation-finite;
  \item[\rm(2)] $A$, $B$, and $C$ both are representation-finite;
  \item[\rm(3)] for any $f_1, f_2, \cdots, f_n \in \{\BD, (-)^{\CMA}\}$,
    $f_1f_2\cdots f_n(A)$ is representation-finite.
\end{itemize}
{\rm(Note that it has been proved in \cite{CL2019} that the representation types of $A$ and $C$ coincide.)}
\end{theorem}
\end{sloppypar}

\begin{sloppypar}
\begin{theorem}[Proposition \ref{prop:der-disc} and Theorem \ref{mainthm:der repr-type}]
Let $A$ be a gentle algebra. Then the following statements are equivalent:
\begin{itemize}
  \item[\rm(1)] one of $A$, $B$ and $C$ is derived-discrete;
  \item[\rm(2)] $A$, $B$, and $C$ both are derived-discrete;
  \item[\rm(3)] for any $f_1, f_2, \cdots, f_n \in \{\BD, (-)^{\CMA}\}$,
    $f_1f_2\cdots f_n(A)$ is derived-discrete.
\end{itemize}
\end{theorem}
\end{sloppypar}


This paper is organized as follows.
In Section \ref{Sect:Prel}, we recall some preliminaries on the geometric models, BD-gentle algebras, and CM-Auslander algebras of gentle algebras, including the constructions of the bound quivers of BD-gentle algebras and CM-Auslander algebras.
Based on the above constructions, we provide a method to draw the marked surfaces of BD-gentle algebras and CM-Auslander algebras in Section \ref{Sect:BD CMA}.
In Section \ref{Sect:repr-type} we prove the main results of this paper which points out that
the (derived) representation types of gentle algebra $A$, BD-gentle algebra $B$, and CM-Auslander algebra $C$ coincide.

\section{Preliminaries} \label{Sect:Prel}

In this section, we recall the definitions and some properties of gentle algebras and their geometric models, Burban--Drozd algebras, and Cohen-Macaulay Auslander algebras.
Throughout this paper, we use $\Q = (\Q_0,\Q_1,\source,\target)$ to represent a finite connected quiver,
where $\Q_0$ is the vertex set, $\Q_1$ is the arrow set,
$\source$ is the function $\Q_1\to\Q_0$ sending each arrow to its source,
and $\target$ is the function $\Q_1\to\Q_0$ sending each arrow to its sink.
It is well-known that the composition of two arrows $a$ and $b$ induces a multiplication
\[ \cdot: \kk\Q \times \kk\Q \to \kk\Q \]
in the path algebra $\kk\Q$.
In this paper, the composition of $a$ and $b$ is defined as follows:
\[
a\cdot b
:= \begin{cases}
 ab, & \text{ if } \target(a) = \source(b);  \\
 0, & \text{ otherwise. }
\end{cases}
\]
$\forall v\in\Q_0$, $\e_v$ is the path of length zero corresponding to $v$.
Let $\Path = a_1\cdots a_n$ ($a_1,\ldots,a_n \in \Q_1$) be a path.
We use $\ell(p)$ to represent the length of $p$, that is, $\ell(p)=n$.
For any finite-dimensional algebra $A=\kk\Q/\I$, we use $\modcat(A)$ to represent the category of finitely generated right $A$-module, see \cite{ASS2006}.

\subsection{Gentle algebras and their geometric models} \label{subsect:geo mods}

Each finite-dimensional hereditary Nakayama algebra $\pmbA_m := \kk\bbA_m$ corresponds to an $(m+1)$-polygon $P=\Poly(\pmbA_m)$ which is shown in \Pic \ref{fig:polygon}, where $\bbA_m$ is the quiver
\[ \xymatrix{
1 \ar[r]^{a_1} & 2 \ar[r]^{a_2} & \cdots \ar[r]^{a_{m-1}} & m \ar[r]^{a_m} & m+1,
} \]
such that:
\begin{itemize}
  \item[(1)] $\Poly$ has a unique edge which is said to be a {\defines boundary component} $\bcomp = \bcomp(\pmbA_m)$;
  \item[(2)] all edges except the boundary component are said to be {\defines $\rbullet$-arcs} (or {\defines closed arcs}) which are marked by $\reda{1}{\Poly}$, $\reda{2}{\Poly}$, $\ldots$, $\reda{m}{\Poly}$ in anticlockwise order such that $\reda{i}{\Poly}$ corresponds to the vertex $i$ of $\bbA_m$ for all $1 \=< i \=< m$.
\end{itemize}
\begin{figure}[htbp]
\centering
\begin{tikzpicture}[scale=1.25]
\draw[ red ][line width = 1pt][rotate=  0] (2,0) -- (1,1.73);
\draw[ red ][line width = 1pt][rotate= 60] (2,0) -- (1,1.73)[dotted];
\draw[ red ][line width = 1pt][rotate=120] (2,0) -- (1,1.73);
\draw[ red ][line width = 1pt][rotate=180] (2,0) -- (1,1.73);
\draw[black][line width = 2pt][rotate=240] (2,0) -- (1,1.73);
\draw[ red ][line width = 1pt][rotate=300] (2,0) -- (1,1.73);
\filldraw[red][line width = 1pt][fill=white][rotate=  0] (2,0) circle(1mm);
\foreach \x in {0,60,120,180,240,300} \draw[red][line width = 1pt][fill=white][rotate= \x] (2,0) circle(1mm);
%
\draw (0,-2.2) node{boundary};
\draw (0,-2.5) node{component $\bcomp(\pmbA_m)$};
\draw[rotate=  60] (0,-2.2) node{$\reda{1}{P}$};
\draw[rotate= 120] (0,-2.2) node{$\reda{2}{P}$};
\draw[rotate= 240] (0,-2.2) node{$\reda{m-1}{P}$};
\draw[rotate= 300] (0,-2.2) node{$\reda{m}{P}$};
\draw[cyan][line width=1pt][rotate=  60] (0.2,-1.73) arc(180:60:0.8);
\draw[cyan][line width=1pt][rotate= 120] (0.2,-1.73) arc(180:60:0.8)[dotted];
\draw[cyan][line width=1pt][rotate= 180] (0.2,-1.73) arc(180:60:0.8)[dotted];
\draw[cyan][line width=1pt][rotate= 240] (0.2,-1.73) arc(180:60:0.8);
\draw[cyan][line width=1pt][rotate=  60][->] (0.2,-1.73) arc(180:115:0.8);
\draw[cyan][line width=1pt][rotate= 120][->] (0.2,-1.73) arc(180:115:0.8)[dotted];
\draw[cyan][line width=1pt][rotate= 180][->] (0.2,-1.73) arc(180:115:0.8)[dotted];
\draw[cyan][line width=1pt][rotate= 240][->] (0.2,-1.73) arc(180:115:0.8);
\draw[cyan] ( 0.9,0) node{$\alpha_1$};
\draw[cyan] (-0.9,0) node{$\alpha_m$};
\fill[blue] (0,-1.73) circle(1mm) node[above]{$\marked(\pmbA_m)$};
\end{tikzpicture}
\caption{The polygon $\redP(\pmbA_m)$ corresponding to $\pmbA_m$}
\label{fig:polygon}
\end{figure}
Then each clockwise interior angle (see $\alpha_1$, $\cdots$, $\alpha_m$ in \Pic \ref{fig:polygon}) decided by $\reda{i}{\Poly}$ and $\reda{i+1}{\Poly}$ ($1 \=< i < m$) corresponds to the arrow $a_i$ of $\bbA_m$.
A {\defines $\rbullet$-polygon} corresponding to $\pmbA_m$ is the pair $\redP(\pmbA_m) := (\Poly(\pmbA_m), \marked(\pmbA_m))$ given by $\Poly=\Poly(\pmbA_m)$ and $\marked(\pmbA_m)$,
where $\marked(\pmbA_m)$ is a point on the boundary component of $\marked(\pmbA_m)$ which is written as ``$\bbullet$'' in \Pic \ref{fig:polygon}.
In particular, we define that $\frakE(\redP(\pmbA_m))$ is defined as the set of all $\rbullet$-arcs of $P$.

Now we recall gentle algebras and their marked surfaces.

Recall that a {\defines monomial algebra} is a finite-dimensional bound quiver algebra $\kk\Q/\I$ such that
the ideal $\I$ is generated by some path of length $\>= 2$.
A {\defines string algebra} is a monomial algebra such that the following conditions hold.
\begin{itemize}
  \item[(1)] $\forall v\in \Q_0$, the number $\source^{-1}(v)$ of arrows starting at $v$ is less than or equal to $2$;
  \item[(2)] $\forall v\in \Q_0$, the number $\target^{-1}(v)$ of arrows ending at $v$ is less than or equal to $2$;
  \item[(3)] $\forall a\in \Q_1$, at most one arrow $b$ with $\target(a)=\source(b)$ satisfies $ab\notin\I$;
  \item[(4)] $\forall b\in \Q_1$, at most one arrow $a$ with $\target(a)=\source(b)$ satisfies $ab\notin\I$.
\end{itemize}

{\defines Gentle algebra} is introduced by Assem and Skowro\'{n}ski in \cite{AS1987}, which is defined as a string algebra satisfying the following conditions.
\begin{itemize}
  \item[(5)] $\forall a\in \Q_1$, at most one arrow $b$ with $\target(a)=\source(b)$ satisfies $ab\in\I$;
  \item[(6)] $\forall b\in \Q_1$, at most one arrow $a$ with $\target(a)=\source(b)$ satisfies $ab\in\I$;
  \item[(7)] all generators of $\I$ are path of length two.
\end{itemize}

In \cite{BD2017}, Burban and Drozd provide another definition of gentle algebra which is equivalent to that of gentle algebra given by Assem and Skowro\'{n}ski.

\begin{definition} \rm \label{def:BD-gentle}
A {\defines gentle algebra} $A$ is defined as the following quotient
\[ A(\bm{m},\simeq_{\gentleSet}) :=
\prod_{i=1}^s H_i \bigg/ \langle E_{jj}^{(i)}-E_{\jmath \jmath}^{(\imath)} \mid (i,j) \simeq_{\gentleSet} (\imath, \jmath)\rangle, \]
where
\begin{itemize}
  \item $\bm{m} = (m_1,m_2,\cdots,m_s) \in \NN^{s}_{\>= 2}$;
  \item $H_i$ is a lower triangular $(m_i\times m_i)$-matrix algebra $T_{m_i}$;
  \item $E_{jj}^{(i)} = (x_{uv})_{1 \=< u,v \=< m_i}$ is the matrix lying in $H_i$
     such that $x_{uv}= \begin{cases}
      1, & u=v=j; \\ 0, & \text{otherwise};
     \end{cases}$
  \item ``$\simeq_{\gentleSet}$'' is a symmetric relation defined on
      \[ \gentleSet := \gentleSet(\bm{m}) := \bigcup_{1 \=< i \=< s} \{(i,j) \mid 1 \=< j \=< m_i\},  \]
      such that for some $(i,j) \in \gentleSet$, there is a unique $(\imath,\jmath) \in \gentleSet$ with $(i,j) \ne (\imath,\jmath)$ satisfying $(i,j) \simeq_{\gentleSet} (\imath,\jmath)$.
\end{itemize}
The direct product $H := \prod_{i=1}^s H_i$ is called {\defines normalization} of $A$.
\end{definition}

The geometric model of a gentle algebra is the marked surface introduced by \cite{HKK2017}.
In \cite{OPS2018} and \cite{BCS2021}, Opper$-$Plamondon$-$Schroll and Baur$-$Coelho-Sim\~{o}es
reestablished the definition of the geometric model for each gentle algebra,
and used the above geometric models to describe the derived category and the module category of gentle algebra, respectively.
In fact, in Definition \ref{def:BD-gentle}, we have $H_i \cong \pmbA_{m_i}$,
and then $H_i$ corresponds to the $\rbullet$-polygon $\Poly(\pmbA_{m_i})$.
This correspondence provides a method to draw the geometric model of any gentle algebra (see \cite{QZZ2022}).
Let us review it.

\begin{definition} \rm \label{def:QZZ}
A {\defines marked surface} $\bfS = (\Surf,\Marked,\rAS)$ is a triple given by the following steps.

(1) $\Surf$ is the surface given by the following steps.
\begin{itemize}
  \item A $t$-tuple $\pmb{\Poly}=(\Poly_1, \ldots, \Poly_s)$ with an involution function $\mathrm{inv}: \prerAS^{\sim} \to \prerAS^{\sim}$,
    where:
    \begin{itemize}
      \item $\Poly_j = \Poly(\pmbA_{m_j})$ and $m_j \>= 2$ hold for all $1 \=< j \=< s$;
      \item $\prerAS^{\sim}$ is some subset of the set $\prerAS$ of all $\rbullet$-arcs,
    \end{itemize}
    such that for some $\rbullet$-arc $\reda{i}{\Poly_j}$ of $\Poly_j$ ($1\=< i\=< m_j$),
    there is a unique $\rbullet$-arc $\reda{\imath}{\Poly_{\jmath}}$ of $\Poly_{\jmath}$ ($1\=< \imath\=< m_{\jmath}$)
    which satisfies $\mathrm{inv}(\reda{i}{\Poly_j}) = \reda{\imath}{\Poly_{\jmath}}$.

  \item A $t$-tuple $\pmb{\PolyII}=(\PolyII_1, \ldots, \PolyII_t)$ with a bijection $\mathrm{bij}: \preAS_{\pmb{\PolyII}} \to \prerAS\backslash\prerAS^{\sim}$, where:
    \begin{itemize}
      \item each $\PolyII_k$ ($1\=< k\=< t$) is the $2$-gon $\redP(\pmbA_{n_k})$ corresponding to the simple algebra $\pmbA_{n_k} = \pmbA_1$;
      \item $\preAS_{\pmb{\PolyII}} = \{ \reda{1}{\PolyII_k} \mid 1 \=< k \=< t \}$.
    \end{itemize}

  \item The surface $\Surf$ is defined as $(\bigcup_{j=1}^s \Poly_j \cup \bigcup_{k=1}^t \PolyII_j)/\simeq$,
    where ``$\simeq$'' is the symmetric relation defined on $\preAS_{\pmb{Q}}\cup\prerAS$ given by $\reda{i}{\Poly_j} \simeq \reda{\imath}{\Poly_{\jmath}}$ ($\reda{i}{\Poly_j} \ne \reda{\imath}{\Poly_{\jmath}}$) if and only if one of the following holds:
    \begin{itemize}
      \item $\mathrm{bij}^{\pm1}(\reda{i}{\Poly_j}) = \reda{\imath}{\Poly_{\jmath}}$;
      \item $\mathrm{inv}(\reda{i}{\Poly_j}) = \reda{\imath}{\Poly_{\jmath}}$.
    \end{itemize}
    That is, if $\reda{i}{\Poly_j} \simeq \reda{\imath}{\Poly_{\jmath}}$, then we glue them.
    In this case, we naturally obtain an equivalence relation defined on the set of all vertices of $\rbullet$-polygons.
\end{itemize}

(2) $\Marked = \rMarked \cup \bMarked$, where
\begin{itemize}
  \item $\rMarked$ is the set of all endpoints of $\rbullet$-arcs up to the symmetric relation ``$\simeq$'' given as above, and the elements in $\rMarked$ are called {\defines $\rbullet$-marked points} or {\defines closed marked points};
  \item $\bMarked = \{\marked(\pmbA_{n_j}) \mid 1\=< j\=< s\} \cup \{\marked(\pmbA_{n_k}) \mid 1\=< k\=< s\}$, and the elements in $\bMarked$ are called {\defines $\bbullet$-marked points} or {\defines open marked points}
\end{itemize}

(3) $\rAS := (\preAS_{\pmb{Q}}\cup\prerAS)/\simeq  \ = \prerAS/\sim$ is the set of all $\rbullet$-arcs (up to the equivalence naturally obtained by ``$\simeq$'',
we provide an instance in Example \ref{exp:1} for it, see (\ref{formula: in exp:1})).
The set $\rAS$ is called a {\defines closed full formal arc system} or a {\defines full formal $\rbullet$-arc system} (=$\bbullet$-FFAS for short).
\end{definition}

The following theorem shows that any gentle algebra has a marked surface.

\begin{theorem}[{\!\!\cite{OPS2018,BCS2021,QZZ2022}}] \label{thm:corresp}
There is a bijection between the set of isoclasses of gentle algebras and the set of homotopy classes of marked surfaces. Definition \ref{def:QZZ} provides a method to draw the marked surface of each gentle algebra.
\end{theorem}

For a marked surface $\bfS$, its $\rbullet$-FFAS $\rAS$ has a dual arc system, write it as $\bAS$ and call it as an {\defines open full formal arc system} or a {\defines full formal $\bbullet$-arc system} (=$\bbullet$-FFAS for short), which can be obtained by the following steps.
\begin{itemize}
  \item[(1)] $\bAS$ is a set of some $\bbullet$-arc. Here, an {\defines $\bbullet$-arc} is a curve whose ending points are $\bbullet$-marked points.
  \item[(2)] For any $\blueanoind \in \bAS$, there is a unique $\redanoind \in \rAS$ such that $\blueanoind \cap \redanoind$ has a unique intersection in the inner $\Surf^{\circ}:=\Surf\backslash(\partial\Surf\cup \rpunc)$ of $\Surf$,
      where $\partial\Surf$ is the boundary of $\Surf$ (that is, the union of all boundary components) and $\rpunc$ is the set of all $\rbullet$-marked points lying in $\Surf\backslash\partial\Surf$.
  \item[(3)] For any $\redanoind \in \rAS$, there is a unique $\blueanoind \in \bAS$ such that $\redanoind \cap \blueanoind$ has a unique intersection in $\Surf^{\circ}$.
\end{itemize}
It is clear that $\bAS$ is uniquely determined by $\rAS$.
Therefore, $\bfS = (\Surf, \Marked, \rAS)$ can be written as $\bfS = (\Surf, \Marked, \bAS)$ in some cases.

\begin{example} \label{exp:1} \rm
Let $A = \kk\Q/\I$ be a gentle algebra given by the bound quiver $(\Q,\I)$, where $\Q = $
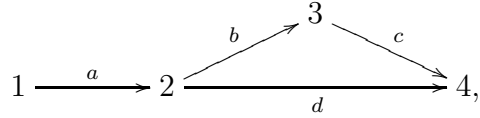
\[ \xymatrix@R=0.4cm@C=1.5cm{ & & 3\ar[rd]^{c} & \\ 1\ar[r]^{a} & 2\ar[ru]^{b} \ar[rr]_{d} & & 4, }\]
and $\I = \langle ad \rangle$.
Then $A \cong (H_1\times H_2) /\langle E_{22}^{(1)}-E_{11}^{(2)}, E_{33}^{(1)}-E_{22}^{(2)} \rangle$, where
\[\gentleSet = \{ (1,1), (1,2), (1,3), (1,4), (2,1), (2,2) \}; \]
\[(1,2) \simeq_{\gentleSet} (2,1), \ (1,4) \simeq_{\gentleSet} (2,2); \]
\[H_1 = T_{4}
= \left(
  \begin{smallmatrix}
   \kk & & & \\
   \kk & \kk & & \\
   \kk & \kk & \kk & \\
   \kk & \kk & \kk & \kk
  \end{smallmatrix}
  \right) \cong \pmbA_3; \]
\[H_2 = T_{2}
= \left(
  \begin{smallmatrix}
   \kk & & \\
   \kk & \kk & \\
   \kk & \kk & \kk
  \end{smallmatrix}
  \right) \cong \pmbA_2. \]
\Pic \ref{fig:examp MS} shows the marked surface $\bfS = (\Surf,\Marked,\rAS)$ of $A$.
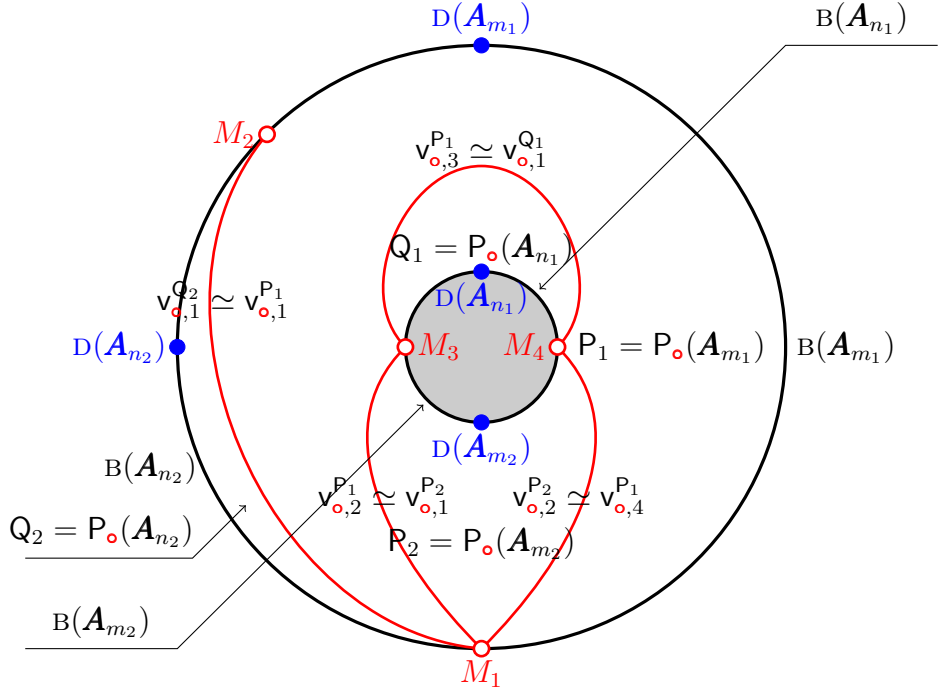
\begin{figure}[htbp]
\begin{center}
\begin{tikzpicture}[scale=2]
\draw [black][line width=1.2pt] (0,0) circle (2cm);
\filldraw [black!20] (0,0) circle (0.5cm);
\draw [black][line width=1.2pt] (0,0) circle (0.5cm);
\draw[red][line width=1pt] (-0.5,0) to[out=135,in=180] (0, 1.2) to[out=0, in=45] (0.5,0);
\draw[red][line width=1pt] (-0.5,0) to[out=-135,in=135] (0,-2);
\draw[red][line width=1pt] (0.5,0) to[out=-45,in=45] (0,-2);
\draw[red][line width=1pt] (-1.41,1.41) to[out=-130,in=175] (0,-2);
\filldraw[blue] (0,2) circle (0.05cm); \draw[blue] (0,2) node[above]{$\marked(\pmbA_{m_1})$};
\filldraw[blue] (-2,0) circle (0.05cm); \draw[blue] (-2,0) node[left]{$\marked(\pmbA_{n_2})$};
\filldraw[blue] (0,0.5) circle (0.05cm); \draw[blue] (0,0.5) node[below]{$\marked(\pmbA_{n_1})$};
\filldraw[blue] (0,-0.5) circle (0.05cm); \draw[blue] (0,-0.5) node[below]{$\marked(\pmbA_{m_2})$};
\filldraw[red][fill=white][line width=1pt] ( 0   ,-2   ) circle (0.05cm) node[below]{$M_1$};
\filldraw[red][fill=white][line width=1pt] (-1.41, 1.41) circle (0.05cm) node[ left]{$M_2$};
\filldraw[red][fill=white][line width=1pt] (-0.5 , 0   ) circle (0.05cm) node[right]{$M_3$};
\filldraw[red][fill=white][line width=1pt] ( 0.5 , 0   ) circle (0.05cm) node[ left]{$M_4$};
\draw[red] (-1.7 , 0.30) node[black]{$\reda{1}{\PolyII_2} \simeq \reda{1}{\Poly_1}$};
\draw[red] (-0.64,-1.00) node[black]{$\reda{2}{\Poly_1} \simeq \reda{1}{\Poly_2}$};
\draw[red] ( 0.64,-1.00) node[black]{$\reda{2}{\Poly_2} \simeq \reda{4}{\Poly_1}$};
\draw[red] ( 0.00, 1.30) node[black]{$\reda{3}{\Poly_1} \simeq \reda{1}{\PolyII_1}$};
\draw[black] (-1.8,-0.8) node[left]{$\bcomp(\pmbA_{n_2})$};
\draw[black] (2,0) node[right]{$\bcomp(\pmbA_{m_1})$};
\draw[black][<-] ( 1.41*0.25+0.03, 1.41*0.25+0.03) -- (2,2) -- (3,2);
\draw[black] (2.5,2) node[above]{$\bcomp(\pmbA_{n_1})$};
\draw[black][<-] (-1.41*0.25-0.03, -1.41*0.25-0.03) -- (-2,-2) -- (-3,-2);
\draw[black] (-2.5,-2) node[above]{$\bcomp(\pmbA_{m_2})$};
\draw (1.25,0) node{$\Poly_1=\redP(\pmbA_{m_1})$};
\draw (0,-1.3) node{$\Poly_2=\redP(\pmbA_{m_2})$};
\draw (0,0.65) node{$\PolyII_1=\redP(\pmbA_{n_1})$};
\draw[<-] (-1.55, -1.05) -- (-1.9,-1.4) -- (-3,-1.4);
\draw (-2.5, -1.4) node[above]{$\PolyII_2=\redP(\pmbA_{n_2})$};
\end{tikzpicture}
\caption{An example for marked surface}
\label{fig:examp MS}
\end{center}
\end{figure}
Here, keep the notations from Definition \ref{def:QZZ}, we have:
\[ s=2, m_1=4, m_2=2, \pmb{\Poly} = \{\Poly_1, \Poly_2\}; \]
\[ \prerAS = \{\reda{1}{\Poly_1}, \reda{2}{\Poly_1}, \reda{3}{\Poly_1}, \reda{4}{\Poly_1};
   \reda{1}{\Poly_2}, \reda{2}{\Poly_2}\} \supseteq
   \prerAS^{\sim} = \{ \reda{2}{\Poly_1}, \reda{4}{\Poly_1}, \reda{1}{\Poly_2}, \reda{2}{\Poly_2} \}; \]
\begin{align*}
  \mathrm{inv} : \prerAS^{\sim} & \to \prerAS^{\sim} \text{ which is defined as } \\
    \reda{2}{\Poly_1} & \mapsto \reda{1}{\Poly_2}, \ \
    \reda{4}{\Poly_1}   \mapsto \reda{2}{\Poly_2}, \\
    \reda{1}{\Poly_2} & \mapsto \reda{2}{\Poly_1}, \ \
    \reda{2}{\Poly_2}   \mapsto \reda{4}{\Poly_1};
\end{align*}
\[ t=2, n_1=n_2=1, \pmb{\PolyII} = \{\PolyII_1, \PolyII_2\}; \]
\[ \preAS_{\pmb{\PolyII}} = \{ \reda{1}{\PolyII_1}, \reda{1}{\PolyII_2} \},
   \prerAS\backslash\prerAS^{\sim} = \{ \reda{1}{\Poly_1}, \reda{3}{\Poly_1} \}
; \]
\begin{align*}
  \mathrm{bij} : \preAS_{\pmb{\PolyII}} & \to \prerAS\backslash\prerAS^{\sim} \text{ which is defined as } \\
    \reda{1}{\PolyII_1} & \mapsto \reda{3}{\Poly_1}, \ \
    \reda{2}{\PolyII_2}   \mapsto \reda{1}{\Poly_1};
\end{align*}
\begin{align} \label{formula: in exp:1}
  \rAS & = \{ \reda{1}{\Poly_1}, \reda{2}{\Poly_1}, \reda{3}{\Poly_1}, \reda{4}{\Poly_1};
           \reda{1}{\Poly_2}, \reda{2}{\Poly_2}, \reda{1}{\PolyII_1}, \reda{1}{\PolyII_2} \}/ \simeq \nonumber \\
       & = \{ \{\reda{1}{\Poly_1}, \reda{1}{\PolyII_2}\},
              \{\reda{2}{\Poly_1}, \reda{1}{\Poly_2}\},
              \{\reda{3}{\Poly_1}, \reda{1}{\PolyII_1}\},
              \{\reda{4}{\Poly_1}, \reda{2}{\Poly_2}\}
           \} \\
       & \text{which is obtained by ``$\simeq$'' dividing $\preAS_{\pmb{\PolyII}} \cup \prerAS$}; \nonumber
\end{align}
and
\[ \Marked = \rMarked \cup \bMarked
           = \{ M_1, M_2, M_3, M_4,
                \marked(\pmbA_{m_1}), \marked(\pmbA_{m_2}),
                \marked(\pmbA_{n_1}), \marked(\pmbA_{n_2}) \}. \]
Notice that each unordered pair $\{\reda{j}{\Poly_i}, \reda{\jmath}{\Poly_{\imath}}\}$ lying in $\rAS$ can be seen an equivalence class naturally obtain by ``$\simeq$''.
The $\bbullet$-FFAS $\bAS$ decided by $\rAS$ is shown in \Pic \ref{fig:FFAS}.
\begin{figure}[htbp]
\begin{center}
\begin{tikzpicture}[scale =2]
\draw [black][line width=1.2pt] (0,0) circle (2cm);
\filldraw [black!20] (0,0) circle (0.5cm);
\draw [black][line width=1.2pt] (0,0) circle (0.5cm);
\draw[blue][line width=1pt] (-2,0) to[out=80,in=-160] (0,2);
\draw[blue][line width=1pt] (0,2) -- (0,0.5);
\draw[blue][line width=1pt]
(0,2) to [out=-135, in=90] (-1.5,0) to[out=-90, in=180] (-0.8,-0.8) to [out=0,in=-135] (0,-0.5);
\draw[blue][line width=1pt]
(0,2) to [out=-45, in=90] (1.5,0) to[out=-90, in=0] (0.8,-0.8) to [out=180,in=-45] (0,-0.5);
\filldraw[blue] (0,2) circle (0.05cm); \draw[blue] (0,2) node[above]{$\marked(\pmbA_{m_1})$};
\filldraw[blue] (-2,0) circle (0.05cm); \draw[blue] (-2,0) node[left]{$\marked(\pmbA_{n_2})$};
\filldraw[blue] (0,0.5) circle (0.05cm); \draw[blue] (0,0.5) node[below]{$\marked(\pmbA_{n_1})$};
\filldraw[blue] (0,-0.5) circle (0.05cm); \draw[blue] (0,-0.5) node[above]{$\marked(\pmbA_{m_2})$};
\filldraw[red][fill=white][line width=1pt] ( 0   ,-2   ) circle (0.05cm) node[below]{$M_1$};
\filldraw[red][fill=white][line width=1pt] (-1.41, 1.41) circle (0.05cm) node[ left]{$M_2$};
\filldraw[red][fill=white][line width=1pt] (-0.5 , 0   ) circle (0.05cm) node[right]{$M_3$};
\filldraw[red][fill=white][line width=1pt] ( 0.5 , 0   ) circle (0.05cm) node[ left]{$M_4$};
\end{tikzpicture}
\caption{The $\rbullet$-FFAS and $\bbullet$-FFAS}
\label{fig:FFAS}
\end{center}
\end{figure}
The marked surface shown in this figure is written as $(\Surf, \Marked, \bAS)$.
\end{example}

\subsection{Baurban$-$Drozd(BD) gentle algebras}

\begin{definition} \label{def:BD} \rm
Let $A=A(\bm{m},\simeq_{\gentleSet})$ be a gentle algebra and $H=H_1\times \cdots\times H_s$ be its normalization.
Its Buarban$-$Drozd gentle algebra, say BD-gentle algebra and write it as $\BD(A)$ for simplification, is defined as
\[
\BD(A)= \left(\begin{matrix}
    A & H\\
    \rad A & H\\
    \end{matrix}\right).
\]
\end{definition}

The normalization of a gentle algebra is also a gentle algebra whose global dimension is two (see \cite[Theorem 3.1 (1) and (3)]{BD2017}).
For a gentle algebra $A=\kk\Q/\I$, let us recall the method to compute the bound quiver of $\BD(A)$ in this subsection. To do this, we need the following lemma.

\begin{lemma} \label{lemm:BD(Am)}
We have an isomorphism $\BD(\pmbA_m) \cong \pmbA_{2m}$ between two algebras.
\end{lemma}

\begin{proof}
Notice that the normalization of $\pmbA_m$ is itself. Then, by Definition \ref{def:BD}, we have
\[
\BD(\pmbA_m)=\left(\begin{matrix}
    \pmbA_m & \pmbA_m\\
    \rad \pmbA_m & \pmbA_m
    \end{matrix}\right).
\]
Assume $\bm{e}_i = \big( 0\ \cdots \ 0\ \mathop{1}\limits^{i}\ 0\ \cdots \ 0  \big)$
Then the matrix
$$\mathbf{P}=\begin{pmatrix}
\bm{e}_{m+1}' & \bm{e}_1' & \bm{e}_{m+2}' & \bm{e}_2' & \cdots & \bm{e}_{2m-1}' & \bm{e}_{2m}'
\end{pmatrix}$$
is invertible and
\[\BD(\pmbA_m)\rightarrow \pmbA_{2m}, \bm{M}\rightarrow \bm{P}^{-1}\bm{M}\bm{P}\]
is a $\kk$-linear isomorphism,
where $\bm{e}_i'$ is the transposition of $\bm{e}_i$.
Moreover, it is also a homomorphism between two algebras.
Thus, $\BD(\pmbA_m)\cong \pmbA_{2m}$.
\end{proof}

\begin{theorem} \label{thm:BD(gent)}
Let $A=A(\bm{m}, \simeq_{\gentleSet})$ be a gentle algebra where $\bm{m}=(m_1, \cdots, m_t)\in \mathbb{N}^t_{\>= 2}$
and $H=H_1\times\cdots\times H_t$ be the normalization of $A$. Then
\begin{align} \label{formula:BD(gent)}
\BD(A) \cong & \BD(H_1)\times\cdots\times\BD(H_s)/\simeq_{\gentleSet} \\
 := & \{(\bm{X}(1),\cdots,\bm{X}(s)) \in T_{2m_1}(\kk) \times \cdots \times T_{2m_s}(\kk)| \nonumber \\
 & X(i)_{jj}=X(\imath)_{\jmath \jmath} \text{ if } (i,j)\simeq_{\gentleSet} (\imath,\jmath) \text{ and } (i,j)\neq (\imath,\jmath) \} \nonumber
\end{align}
is gentle. Here, for any $1\=< i\=< s$ and $1\=< j\=< 2m_i$, $X(i)_{jj}$ is the element of the $j$-th row and the $j$-th column of the matrix $\bm{X}(j) := (X(j)_{rs})_{1\=<r,d\=<m_j}$.
\end{theorem}

\begin{proof}
For any $r\=< s$, let $\bm{m}_{\=< r}=(m_1, \cdots, m_r)\in \mathbb{N}^r_{\>= 2}$,
$H_{\=< r}=H_1\times\cdots\times H_r$, $A_{\=< r}=A(\bm{m}_{\=< r}, \simeq_{\gentleSet})$,
and $B_r=\BD(A_r)$.
For the case of $r<s$, denote by
$B_{\=< r}=\begin{pmatrix}
A_{\=< r} & H_{\=< r}\\
\rad A_{\=< r} & H_{\=< r}\\
\end{pmatrix}.$
Then we have
$$B_{\=< r+1}=\begin{pmatrix}
A_{\=< r+1} & H_{\=< r+1}\\
\rad A_{\=< r+1} & H_{\=< r+1}\\
\end{pmatrix}=\begin{pmatrix}
A_{\=< r}       &                     &    H_{\=< r}   &        \\
                 &  \pmbA_{m_{r+1}}       &                 & H_{r+1}\\
\rad A_{\=< r}  &                     &    H_{\=< r}   &        \\
                 & \rad \pmbA_{m_{r+1}}   &                 & H_{r+1}\\
\end{pmatrix},$$
and there may exist some elements, write them as $(i,j)$, in $\gentleSet(\bm{m}_{\=< r})$ satisfying $(i,j) \simeq_{\gentleSet} (r+1,l)$.
Assume that the matrix
\[
\bm{P}=\begin{pmatrix}
\bm{E}_1 & & &  \\
& \bm{0} & \bm{E}_2 &  \\
& \bm{E}_3 & \bm{0} &  \\
& & & \bm{E}_4 \\
\end{pmatrix},
\]
where $\bm{E}_1 = \bm{E}_4 \in T_{2m_{\=< r}}(\kk)$ ($m_{\=< r} := m_1+\cdots +m_r$) and
$\bm{E}_2 = \bm{E}_3 \in T_{2m_{r+1}}(\kk)$ are identities.
Then we have an isomorphism between two algebras as follows:
\begin{align*}
B_{\=< r+1} \cong
& \bm{P}^{-1}B_{\=< r+1}\bm{P}=
{\begin{pmatrix}
   A_{\=< r} & H_{\=< r} & &  \\
   \rad A_{\=< r} &  H_{\=< r}   &  &  \\
   & & \pmbA_{m_{r+1}}  & H_{r+1} \\
   & & \rad \pmbA_{m_{r+1}} & H_{r+1} \\
\end{pmatrix}
} \\
= &
{\begin{pmatrix}
  B_{\=< r} &\\
  &  B_{r+1} \\
\end{pmatrix}}
= \BD(H_1) \times \cdots \times \BD(H_{r+1})/\simeq_{\gentleSet}.
\end{align*}
Note that $B_{\=< 1}=B_1$, then an easy induction gives (\ref{formula:BD(gent)}) as required.
\end{proof}

In Section \ref{Sect:BD CMA}, we provide a method to draw the marked surface of $\BD(A)$.

\subsection{Cohen$-$Macaulay$-$Auslander (CM-Auslander) algebras} \label{subsect:CMA}
For any ring $R$, recall that a {\defines Gorenstein project $R$-module} (=G-projective for short) is an $R$-module $G$ with a complete projective resolution
\[ \xymatrix{
\cdots \ar[r]^{d_{-2}} & P_{-1}  \ar[r]^{d_{-1}} & P_0 \ar[r]^{d_{0}} & P_1 \ar[r]^{d_{1}} & P_2 \ar[r]^{d_{2}} & \cdots
} \]
such that the above exact sequence is $\Hom_R(-,R)$-exact. Obviously, any projective module is G-projective.
Kalck provided a description of any G-projective module over gentle algebra in \cite{Kal2015}, see the following theorem.

\begin{theorem}[{\!\!\cite[Theorem 2.4 (a)]{Kal2015}}] \label{thm:Kalck Gproj}
Let $A=\kk\Q/\I$ be a gentle algebra. Then an indecomposable module $G$ over $A$ is G-projective if and only if one of the following conditions holds.
\begin{itemize}
  \item[\rm(1)] $G$ is projective;
  \item[\rm(2)] $G\cong \alpha A$, where $\alpha$ is an arrow on some forbidden cycle,
    where a {\defines forbidden cycle} is a path $\scrC = a_0a_1\cdots a_{\ell-1}$ on the bound quiver $(\Q,\I)$ such that $a_{\overline{i}}a_{\overline{i+1}} \in \I$ holds for all $i\in \NN$ {\rm(}$\overline{i}$ is $i$ modulo $\ell$ here{\rm)}.
\end{itemize}
\end{theorem}

This result has been extended to the case of $R$ to be monomial, see \cite{CSZ2018}.
In \cite{LGH2024}, the authors provide another description of all G-projective modules over gentle algebra by using the marked surface of gentle algebra and reprove the theorem given by \cite{GR2005} that the self-injective dimension of gentle algebra is finite.

The Cohen$-$Macaulay$-$Auslander algebra (=CM-Auslander algebra for short) $R^{\CMA}$ of $R$
is introduced by Beligiannis in \cite{B2011}, which is defined as the opposite algebra
\[ R^{\CMA} = \bigg(\End_R\bigg(\bigoplus_{G \in \ind(\Gproj(R))} G\bigg)\bigg)^{\op}, \]
where $\Gproj(R)$ is the set of all G-projective $R$-modules (up to isomorphism).
In \cite{CL2017,CL2019,LZhang2023CM-Auslander}, the authors considered the CM-Auslander algebras for gentle, skew-gentle, and string algebras, and used them to study the representation types of algebras.

Let $A=\kk\Q/\I$ be a gentle algebra. Now we recall a method for computing the bound quiver of the CM-Auslander algebra of gentle algebra that is given by Chen and Lu in \cite[Section 3]{CL2019}.
We define $(\Q^{\CMA},\I^{\CMA})$ as the bound quiver decided by
the quiver $\Q^{\CMA}=(\Q^{\CMA}_0, \Q^{\CMA}_1, \source^{\CMA}, \target^{\CMA})$ and the ideal $\I^{\CMA}$, which are obtained by the following construction.
\begin{itemize}
  \item[(1)] The set of vertices is $\Q_0^{\CMA} = \Q_0 \bigcup \ind(\Gproj(A))$.
  \item[(2)] The set of arrows is  $\Q_1^{\CMA} = \Q_1^{\cyc,+}\cup\Q_1^{\cyc,-}\cup \Q_1^{\ncyc}$, where:
    \begin{itemize}
      \item $\Q_1^{\cyc}$ is the set of all arrows on forbidden cycles;
      \item $\Q_1^{\cyc,+} = \{\alpha^+ \mid \alpha\in\Q_1^{\cyc}\}$;
      \item $\Q_1^{\cyc,-} = \{\alpha^- \mid \alpha\in\Q_1^{\cyc}\}$;
      \item $\Q_1^{\ncyc} = \Q_1\backslash\Q_1^{\cyc}$.
    \end{itemize}
  \item[(3)] For any $a\in\Q_1^{\CMA}$, we have
    \[
       \source^{\CMA}(a) =
       \begin{cases}
        \source(a), & a \in \Q_1^{\ncyc};  \\
        \source(\alpha), & a = \alpha^{-} \in \Q_1^{\cyc,-}; \\
        \alpha A, & a = \alpha^{+} \in \Q_1^{\cyc,+}.
       \end{cases}
    \]
  \item[(4)] For any $b\in\Q_1^{\CMA}$, we have
       \[
       \target^{\CMA}(b) =
       \begin{cases}
        \target(b), & b \in \Q_1^{\ncyc};  \\
        \target(\beta), & b = \beta^{+} \in \Q_1^{\cyc,+}; \\
        \alpha A, & b = \alpha^{-} \in \Q_1^{\cyc,-}.
       \end{cases}
       \]

  \item[(5)] the set of relations $\I^{\CMA}:= \{\beta^+ \alpha ^- \; | \; \beta\alpha\in \I \text{ with } \alpha,\beta\in \Q_1^{\cyc}\}\bigcup \{\beta \alpha  \; | \; \beta\alpha\in \I \text{ with } \alpha,\beta\in \Q_1^{\ncyc}\}$.
\end{itemize}

The following theorem was first proved by Chen and Lu.

\begin{theorem}[{\!\!\cite[Theorem 3.5]{CL2019}}] \label{thm:CL}
Let $A=\kk\Q/\I$ be a gentle algebra and $A^{\CMA}$ be its CM-Auslander algebra.
Then $A^{\CMA} \cong \kk\Q^{\CMA}/\I^{\CMA}$ is a gentle algebra.
\end{theorem}

In Section \ref{Sect:BD CMA}, we provide a method to draw the marked surface of $A^{\CMA}$.

\section{Marked surfaces of BD/CM-Auslander algebras}
\label{Sect:BD CMA}

In this section, we provide a method for drawing the marked surfaces of BD-gentle algebras and CM-Auslander algebras.

\subsection{Some notations in marked surfaces}
For a marked surface $\bfS = (\Surf, \Marked, \rAS) = (\Surf, \Marked, \bAS)$,
if it has a $\rbullet$-marked point $M$ in $\rpunc$, then $M$ can be seen as a special boundary component without $\rbullet$-marked points.
In this case, if we ignore $\rbullet$-FFAS, and the gentle algebra corresponding to $\bfS$ has an infinite global dimension,
then $\rbullet$-FFAS cuts $\Surf$ to get at least one $\bbullet$-elementary polygon of the form shown in \Pic \ref{fig:infty poly}.
This polygon is called an {\defines $\infty-$elementary polygon}.
In \cite{LGH2024}, the authors show that the global dimension of a gentle algebra $A$ is infinite if and only if the marked surface of $A$ has at least one $\infty-$elementary polygon.

\begin{figure}[htbp]
\centering
\begin{tikzpicture}
\foreach \x in {0, 60, 120, ..., 300}
\draw[blue][rotate = \x] ( 1.73,-1.00) -- ( 1.73, 1.00) [line width=1pt];
\foreach \x in {0, 60, 120, ..., 300}
\fill[blue][rotate = \x] ( 1.73,-1.00) circle(1mm);
\foreach \x in {0, 60, 120, ..., 300}
\draw[red][rotate = \x][line width=1pt] (0,0)--(2,0);
\draw[red][fill= white][line width=1pt] (0,0) circle(1mm);
\end{tikzpicture}
\ \
\begin{tikzpicture}
\draw[white] (0,-2) -- (0,2);
\draw (0,0) node{$\Longrightarrow$};
\end{tikzpicture}
\ \
\begin{tikzpicture}
\foreach \x in {0, 60, 120, ..., 300}
\draw[blue][rotate = \x] ( 1.73,-1.00) -- ( 1.73, 1.00) [line width=1pt];
\foreach \x in {0, 60, 120, ..., 300}
\fill[blue][rotate = \x] ( 1.73,-1.00) circle(1mm);
\foreach \x in {0, 60, 120, ..., 300}
\draw[red!50][rotate = \x][line width=1pt][dotted] (0,0)--(2,0);
\draw[black][fill= black!30][line width=1pt] (0,0) circle(5mm);
\end{tikzpicture}
\caption{$\infty-$elementary polygon}
\label{fig:infty poly}
\end{figure}
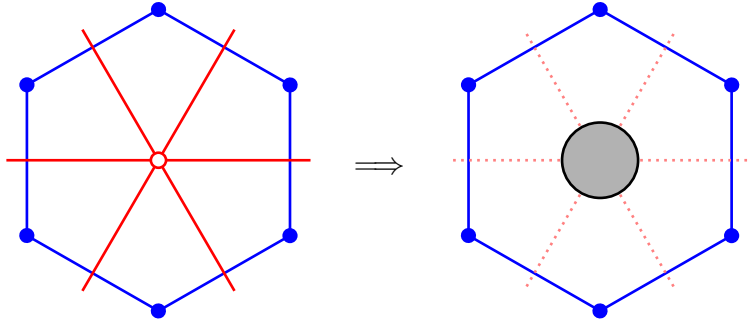

The triple $(\Surf, \Marked, \rAS)$ is used to describe the derived category of gentle algebra in the works of Opper$-$Plamondon$-$Schroll \cite{OPS2018},
and the triple $(\Surf, \Marked, \bAS)$ is used to describe the module category of gentle algebra in the works of Baur$-$Coelho-Sim\~{o}es \cite{BCS2021}.
The $\rbullet$-FFAS $\rAS$ and $\bbullet$-FFAS $\bAS$ are mutually and uniquely determined from each other,
and then we always write the marked surface as $\bfS=(\Surf,\Marked,\AS)$ in this paper.
Here, $\AS := \rAS \cup \bAS$ (for the case of ignoring $\rAS$, we take $\AS := \rAS$).
By Theorem \ref{thm:corresp}, any marked surface can be seen as a geometric model of some gentle algebra $A=\kk\Q/\I$.
Thus, for simplification, we use $\bfS_A=(\Surf_A, \Marked_A, \AS_A)$ to represent it.
In contrast, for a marked surface $\bfS$, we use $A(\bfS)$ to represent the gentle algebra that corresponds to it.

The $\bbullet$-FFAS of a marked surface $\bfS$ divides the surface $\Surf$ into some polygons.
These polygons are called {\defines $\bbullet$-elementary polygons}.
The vertices of all $\bbullet$-elementary polygons are $\bbullet$-marked points.
We can define {\defines $\rbullet$-elementary polygon} in a dual way.

Now, we define a relation ``$\simeq_{\AS}$'' on $\AS$ such that
$c\simeq_{\AS} c'$ if and only if one of the following conditions holds:
\begin{itemize}
  \item $c=c'$;
  \item $c \in \rAS, c'\in \bAS, c\cap c' \cap \Surf^{\circ} \ne \varnothing$;
  \item $c \in \bAS, c'\in \rAS, c\cap c' \cap \Surf^{\circ} \ne \varnothing$.
\end{itemize}
Then ``$\simeq_{\AS}$'' is an equivalence relation, and each equivalence class in $\AS$
is a of the form $\{\redanoind, \blueanoind\}$.
If two arcs $\redanoind$ and $\blueanoind$ are in the same equivalence class,
we say $\redanoind = \blueanoind^{\perp}$ and $\blueanoind = \redanoind^{\perp}$.
If $\rAS$ is ignored, then admits that ``$\simeq_{\AS}$'' is trivial.

The equivalence relation ``$\simeq_{\AS}$'' induces a bijection
\[ \av : \AS/\simeq_{\AS} \ = \{ \{\redanoind, \blueanoind\} \mid \redanoind\in\rAS \} \to \Q_0 \]
from $\AS/\simeq_{\AS}$ to $\Q_0$ (if $\rAS$ is ignored, then $\AS/\simeq_{\AS} = \AS$).
Naturally, the above bijection $\av$ can be seen as a map from $\rAS$ to $\Q_0$ or a map from $\bAS$ to $\Q_0$.
To be precise, for any $\redanoind$ (resp., $\blueanoind$), we have
\[ \av(\redanoind) = \av(\{\redanoind, \redanoind^{\perp}\})
 \ (\text{resp. } \av(\blueanoind) = \av(\{\blueanoind^{\perp}, \redanoind\} )\]
in the above perspective.

\subsection{Marked surfaces of BD gentle algebras}

For each gentle algebra $A$, we can draw the marked surface $\bfS_{\BD(A)}$ of its BD-gentle algebra $\BD(A)$ by Theorem \ref{thm:BD(gent)}. To do this, we construct a marked surface $\BD(\bfS)$ for each $\bfS$ by the following steps.

\begin{construction} \label{const:BDsurface} \rm \

Step 1. Consider the normalization $H=H_1\times\cdots\times H_s$ of $A=A(\bm{m},\simeq_{\gentleSet})$,
then for each $1\=< i\=< s$, the quiver $\Q_{H_i}$ of $H_i$ is
\[ \bbA_{m_i} = \ \ (i,1)\rightarrow (i,2)\rightarrow \cdots \rightarrow (i,m_i), \]
and its marked ribbon surface $\Surf_{H_i}=(\Surf_{H_i},\Marked_{H_i},\AS_{H_i})$ can be obtained by $H_i\cong \pmbA_{m_i}$, see the first picture of \Pic \ref{fig:Hi MS}.
\begin{figure}
\centering
\begin{tikzpicture}[scale = 1.25]
\draw[line width=1pt] (0,0) circle(2cm);
\draw[line width=1pt] ( 0.00, 2.00) -- ( 2.00, 0.00)[blue];
\draw[line width=1pt] ( 0.00, 2.00) -- ( 1.41,-1.41)[blue];
\draw[line width=1pt] ( 0.00, 2.00) -- ( 0.00,-2.00)[blue!25][dotted];
\draw[line width=1pt] ( 0.00, 2.00) -- (-1.41,-1.41)[blue];
\draw[line width=1pt] ( 0.00, 2.00) -- (-2.00, 0.00)[blue];
\draw[blue][rotate around={ -2*\rotateangle:(0,2)}] (1.5,2)
  node{\tiny\rotatebox{-45}{$(i,1)$}};
\draw[blue][rotate around={ -3*\rotateangle:(0,2)}] (1.5,2)
  node{\tiny\rotatebox{-70}{$(i,2)$}};
\draw[blue][rotate around={ -4.75*\rotateangle:(0,2)}] (1.5,2)
  node{\tiny\rotatebox{-113}{$(i,m_i-1)$}};
\draw[blue][rotate around={ -5.65*\rotateangle:(0,2)}] (1.5,2)
  node{\tiny\rotatebox{-132}{$(i,m_i)$}};
\foreach \x in {0,-45,-135,-180}
\draw[line width=1pt][rotate=\x]  ( 1.84, 0.74) -- ( 1.84,-0.74) [red];
\draw[line width=1pt][rotate=-90] ( 1.84, 0.74) -- ( 1.84,-0.74) [red!50][dotted];
\fill[blue]    ( 0.00, 2.00) circle(\pointsize);
\fill[blue]    ( 2.00, 0.00) circle(\pointsize);
\fill[blue]    ( 1.41,-1.41) circle(\pointsize);
\fill[blue!25] ( 0.00,-2.00) circle(\pointsize);
\fill[blue]    (-2.00, 0.00) circle(\pointsize);
\fill[blue]    (-1.41,-1.41) circle(\pointsize);
\foreach \x in {157.5,22.5,-22.5,-67.5,-157.5}
\draw[red][line width = 1pt][fill=white][rotate= \x] (2,0) circle(\pointsize);
\draw[red!50][line width = 1pt][fill=white][rotate=-90-22.5] (2,0) circle(\pointsize);
\end{tikzpicture}
\ \ \ \
\begin{tikzpicture}[scale = 1.25]
\foreach \x in {0,-45,...,-180}
\draw[black!25][fill=black!25][rotate=\x]
  ( 1.96, 0.38) to[out=  180, in=  180]
  ( 1.96,-0.38) to[out=78.75, in=-78.75]
  ( 1.96, 0.38);
\draw[line width=1pt] (0,0) circle(2cm);
\draw[line width=1.0pt] ( 0.00, 2.00) -- ( 2.00, 0.00)[blue];
\draw[line width=0.5pt] ( 0.00, 2.00) -- ( 1.84,-0.74)[blue][dashed];
\draw[line width=1.0pt] ( 0.00, 2.00) -- ( 1.41,-1.41)[blue];
\draw[line width=0.5pt] ( 0.00, 2.00) -- ( 0.74,-1.84)[blue][dashed];
\draw[line width=1.0pt] ( 0.00, 2.00) -- ( 0.00,-2.00)[blue!25][dotted];
\draw[line width=0.5pt] ( 0.00, 2.00) -- (-0.74,-1.84)[blue!25][dashed];
\draw[line width=1.0pt] ( 0.00, 2.00) -- (-1.41,-1.41)[blue];
\draw[line width=0.5pt] ( 0.00, 2.00) -- (-1.84,-0.74)[blue][dashed];
\draw[line width=1.0pt] ( 0.00, 2.00) -- (-2.00, 0.00)[blue];
\draw[line width=0.5pt] ( 0.00, 2.00) -- (-1.84, 0.74)[blue][dashed];
\draw[blue][rotate around={ -2.4*\rotateangle:(0,2)}] (2,2)
  node{\tiny\rotatebox{-62}{$(i,1)'$}};
\draw[blue][rotate around={ -3.35*\rotateangle:(0,2)}] (2,2)
  node{\tiny\rotatebox{-85}{$(i,2)'$}};
\draw[blue][rotate around={ -5.11*\rotateangle:(0,2)}] (2,2)
  node{\tiny\rotatebox{-123}{$(i,m_i-1)'$}};
\draw[blue][rotate around={ -6.05*\rotateangle:(0,2)}] (1.25,2)
  node{\tiny\rotatebox{-147}{$(i,m_i)'$}};
\foreach \x in {0,-45,-135,-180}
\draw[red][rotate=\x][line width=1pt] ( 1.96, 0.38) to[out=180,in=180] ( 1.96,-0.38);
\foreach \x in {-22.5,-67.5,-157.5,-202.5}
\draw[red][rotate=\x][line width=1pt] ( 1.96, 0.38) to[out=180,in=180] ( 1.96,-0.38)[dashed];
\draw[red!50][rotate=-112.5][line width=1pt] ( 1.96, 0.38) to[out=180,in=180] ( 1.96,-0.38)[dashed];
\draw[red!50][rotate= -90.0][line width=1pt] ( 1.96, 0.38) to[out=180,in=180] ( 1.96,-0.38)[dotted];
\fill[blue]    ( 0.00, 2.00) circle(\pointsize);
\fill[blue]    ( 2.00, 0.00) circle(\pointsize);
\fill[blue]    ( 1.41,-1.41) circle(\pointsize);
\fill[blue!25] ( 0.00,-2.00) circle(\pointsize);
\fill[blue]    (-2.00, 0.00) circle(\pointsize);
\fill[blue]    (-1.41,-1.41) circle(\pointsize);
\foreach \x in {157.5,-22.5,-67.5,-157.5}
\fill[blue][line width = 1pt][rotate= \x] (2,0) circle(\pointsize);
\fill[blue!25][line width = 1pt][rotate=-90-22.5] (2,0) circle(\pointsize);
\foreach \x in {11.25,-11.25,-33.75,-56,-78.5,-146,-168.5,-191.25,-213.75}
\draw[red][line width = 1pt][fill=white][rotate= \x] (2,0) circle(\pointsize);
\draw[red!50][line width = 1pt][fill=white][rotate=-101] (2,0) circle(\pointsize);
\draw[red!50][line width = 1pt][fill=white][rotate=-123.5] (2,0) circle(\pointsize);
\end{tikzpicture}
\caption{The marked surfaces of $H_i$ and $\BD(H_i)$}
\begin{center}
  (each $\bbullet$-arc lying in $\av^{-1}((i,j))$ is marked ``$(i,j)$'' in this figure,

  dashed lines are newly added arcs when changing from $H_i$ to $\BD(H_i)$,

  and the dotted line represents some omitted arcs.)
\end{center}
\label{fig:Hi MS}
\end{figure}

Step 2. For $(i,j)\in \gentleSet(\bm{m})$, if there exists $(k,l)\in \gentleSet(\bm{m})$
  such that $(i,j)\simeq_{\gentleSet} (k,l)$,
  then we consider the marked surface $\Surf_{\BD(H_i)}$ which is shown in the second picture of \Pic \ref{fig:Hi MS}.
  Here, we have $\BD(H_i) \cong \pmbA_{2m_i}$ by Lemma \ref{lemm:BD(Am)},
  and the quiver $\Q_{\BD(H_i)}$ of $\BD(H_i)$ is
    \[ \bbA_{2m_i} = \ \ (i,1)'\rightarrow (i,1)\rightarrow (i,2)' \rightarrow (i,2) \rightarrow \cdots \rightarrow (i,m_i)' \rightarrow (i,m_i), \]
  Removing the elementary 2-gon with $\rbullet$-arc $(\av^{-1}(i,j))^\perp$ as an edge of it (see the shadow parts in the second picture of \Pic \ref{fig:Hi MS}),
  we obtain a surface $\Surf^{\rm{inc}}_{\BD(H_i)}$ which is said to be {\defines incomplete}.

Step 3. For all $(i,j) \simeq_{\gentleSet} (k,l)\in \gentleSet(\bm{m})$,
gluing the $\rbullet$-arc $(\av^{-1}(i,j))^\perp$ in $\bfS^{\rm{inc}}_{\BD(H_i)}$ and $\rbullet$-arc $(\av^{-1}(k,l))^\perp$ in $\bfS^{\rm{inc}}_{\BD(H_i)}$,
we obtain a new surface $\BD(\bfS)$.
\end{construction}

\begin{proposition} \label{prop:BDsurface} \
\begin{itemize}
  \item[\rm(1)] Let $A$ be a gentle algebra and $\bfS_A$ be its marked surface.
    Then $\bfS_{\BD(A)}$ is homotopic to $\BD(\bfS_A)$.
  \item[\rm(2)] Let $\bfS$ be a marked surface and $A(\bfS)$ be its gentle algebra.
    Then $A(\BD(\bfS)) \cong \BD(A(\bfS))$.
\end{itemize}
\end{proposition}

\begin{proof}
We only prove (1); the statement (2) is a direct corollary of the statement (1) and Theorem \ref{thm:corresp}.

Assume that the normalization of $A$ is $H=H_1\times \cdots \times H_s$,
then the normalization of $\BD(A)$ is $\BD(H) = \BD(H_1)\times \cdots \times \BD(H_s)$,
see the isomorphism (\ref{formula:BD(gent)}) given in Theorem \ref{thm:BD(gent)}.
Then, by Definition \ref{def:QZZ} and Theorem \ref{thm:corresp},
$\bfS_A$ is structured by all $\rbullet$-elementary polygons $\redP(H_i)$ ($1\=< i\=< s$)
(take $m=m_i$ and $\pmbA_m=H_i$ in \Pic \ref{fig:polygon}),
and $\bfS_{\BD(A)}$ is structured by all $\rbullet$-elementary polygons $\redP(\BD(H_i))=\redP(\pmbA_{2m_i})$.
Here, $\redP(\BD(H_i))$ is shown in \Pic \ref{fig:BD(polygon)}.
\begin{figure}[htbp]
\centering
\begin{tikzpicture}[scale=1.25]
\draw (-1,-1.73) -- (1,-1.73) [line width = 1.5pt];
\foreach \x in {-30,30,150,210}
\draw[ red ][line width = 1pt][rotate=  \x] (2,0) -- (1.73,1);
\foreach \x in {-60,0,120,180}
\draw[ red ][line width =.5pt][rotate=  \x] (2,0) -- (1.73,1) [dashed];
\draw[ red ][line width = 1pt][rotate=  60] (2,0) -- (1.73,1) [dotted];
\draw[ red ][line width = 1pt][rotate=  90] (2,0) -- (1.73,1) [dotted];
\filldraw[red][line width = 1pt][fill=white][rotate=  0] (2,0) circle(1mm);
\foreach \x in {-60,-30,...,240}
\draw[red][line width = 1pt][fill=white][rotate= \x] (2,0) circle(1mm);
%
\draw[rotate=  30] (0.5,-2.2) node{$\reda{1'}{P}$};
\draw[rotate=  60] (0.5,-2.2) node{$\reda{1}{P}$};
\draw[rotate=  90] (0.5,-2.2) node{$\reda{2'}{P}$};
\draw[rotate= 120] (0.5,-2.2) node{$\reda{2}{P}$};
\draw[rotate= 210] (0.6,-2.4) node{$\reda{(m_i-1)'}{P}$};
\draw[rotate= 240] (0.6,-2.4) node{$\reda{m_i-1}{P}$};
\draw[rotate= 270] (0.5,-2.2) node{$\reda{m_i'}{P}$};
\draw[rotate= 300] (0.5,-2.2) node{$\reda{m_i}{P}$};
\fill[blue] (0,-1.73) circle(1mm) node[below]{$\marked(\pmbA_m)$};
\end{tikzpicture}
\caption{The polygon $\redP(\pmbA_m)$ corresponding to $\pmbA_m$}
\label{fig:BD(polygon)}
\end{figure}
In this case, for each edge $\redanoind \in \AS_{\BD(A)}$ of $\redP(\BD(H_i))$,
if for any $\rbullet$-arc $\redanoind'$ such that \
$\redanoind' \ne \redanoind'$ and $\av^{-1}(\redanoind) \simeq_{\gentleSet} \av^{-1}(\redanoind')$,
then we construct an elementary $2$-gon which has only one edge $\mathsf{w}_{\rbullet}$ such that
$\mathsf{w}_{\rbullet}$ is a $\rbullet$-arc and $\redanoind \simeq_{\AS} \mathsf{w}_{\rbullet}$.
We obtain an incomplete surface which is homotopic to $\bfS_{\BD(H_i)}^{\inc}$,
and $\bfS_{\BD(A)}$ is the marked surface obtained by gluing all incomplete surfaces as above.
By construction \ref{const:BDsurface}, we see that $\BD(\bfS_A)$ and $\bfS_{\BD(A)}$ are homotopic.
\end{proof}

\begin{example} \rm \label{exp:BD-alg}
Let $\bm{m} = (2, 3)$, then $\gentleSet(\bm{m})=\{(1,1), (1,2), (2,1), (2,2), (2,3)\}$.
Suppose that the symmetric relation ``$\simeq_{\gentleSet}$'' is given by $(1,1)\simeq_{\gentleSet} (2,2)$,
then we have two path algebras $H_1 \cong \pmbA_2 $ and $H_2 = \pmbA_3$ whose quivers are
\[ (1,1) \To{\alpha}{} (1,2) \]
and
\[ (2,1) \To{\beta}{} (2,2) \To{\gamma}{} (2,3),  \]
and the marked surfaces $\bfS_{\BD(H_1)}$ and $\bfS_{\BD(H_2)}$
are shown in \Pic \ref{fig:BD-alg} I and II, respectively.
\begin{figure}[htbp]
\begin{tikzpicture}[scale=1.3]
\fill [left color = cyan!50] (-1.42,-0.47) arc(198:-18:1.5);
\draw [cyan]    ( 0.00, 0.30) node[right]{$\bfS_{\BD(H_1)}^{\inc}$};
\draw [line width=1.5pt] (0,0) circle (1.5cm);
\draw [blue] (0,1.5)-- (-1.5, 0   ) [line width=.5pt][dashed];
\draw [blue] (0,1.5)-- ( 1.5, 0   ) [line width=.5pt][dashed];
\draw [blue] (0,1.5)-- ( 0  ,-1.5 ) [line width= 1pt];
\draw [blue] (0,1.5)-- (1.11, 1.01) [line width= 1pt];
\draw [red] (-1.11, 1   )-- (-1.42,-0.47) [line width=.5pt][dashed];
\draw [red] (-1.42,-0.47)-- ( 1.42,-0.47) [line width= 1pt];
\draw [red] ( 0.5 , 1.42)-- ( 1.41, 0.51) [line width= 1pt];
\draw [red] ( 1.41, 0.51)-- ( 1.42,-0.47) [line width=.5pt][dashed];
\begin{scriptsize}
\fill [color=blue] ( 0  , 1.5) circle (2.0pt);
\fill [color=blue] (-1.5, 0  ) circle (2.0pt);
\fill [color=blue] ( 1.5, 0  ) circle (2.0pt);
\fill [color=blue] ( 0  ,-1.5) circle (2.0pt);
\fill [color=blue] ( 1.1, 1  ) circle (2.0pt);
\draw [color=red][fill=white] ( 1.41, 0.51) circle (2.0pt) [line width=1pt];
\draw [color=red][fill=white] ( 1.42,-0.47) circle (2.0pt) [line width=1pt];
\draw [color=red][fill=white] (-1.42,-0.47) circle (2.0pt) [line width=1pt];
\draw [color=red][fill=white] (-1.11, 1   ) circle (2.0pt) [line width=1pt];
\draw [color=red][fill=white] ( 0.5 , 1.42) circle (2.0pt) [line width=1pt];
\end{scriptsize}
\draw (0,-1.8) node{$\mathrm{I}.\ \bfS_{\BD(H_1)}$};
\end{tikzpicture}
\ \ \ \ \
\begin{tikzpicture}[scale=1.3]
\fill [left color = green!50] (-1.42,-0.47) arc(198:342:1.5);
\draw [Green]    ( 0.00,-0.70) node[right]{$\bfS_{\BD(H_2)}^{\inc}$};
\draw [line width=1.5pt] (0,0) circle (1.5cm);
\draw [blue] (0,-1.5)-- ( 0  ,  1.5 ) [line width = 1pt];
\draw [blue] (0,-1.5)-- ( 1.11,-1   ) [line width = 1pt];
\draw [blue] (0,-1.5)-- (-1.12,-1   ) [line width = 1pt];
\draw [blue] (0,-1.5)-- (-1.32,-0.72) [line width =.5pt][dashed];
\draw [blue] (0,-1.5)-- ( 1.34,-0.68) [line width =.5pt][dashed];
\draw [blue] (0,-1.5)-- ( 0.7 ,-1.33) [line width =.5pt][dashed];
\draw [red] (-1.42,-0.47)-- (1.42,-0.47) [line width =1pt];
\draw [red] (-1.42,-0.47) to[out=   0, in=  45] (-1.22,-0.88)[line width =.5pt][dashed];
\draw [red] (-1.22,-0.88) to[out= -20, in= 135] (-0.47,-1.42)[line width = 1pt];
\draw [red] ( 0.47,-1.42) to[out=  80, in= 180] ( 0.93,-1.18)[line width =.5pt][dashed];
\draw [red] ( 0.93,-1.18) to[out=  90, in= 180] ( 1.23,-0.86)[line width = 1pt];
\draw [red] ( 1.23,-0.86) to[out=  90, in= 180] ( 1.42,-0.47)[line width =.5pt][dashed];
\begin{scriptsize}
\fill [color=blue] ( 0   ,-1.5 ) circle (2.0pt);
\fill [color=blue] ( 0   , 1.5 ) circle (2.0pt);
\fill [color=blue] ( 1.11,-1   ) circle (2.0pt);
\fill [color=blue] (-1.12,-1   ) circle (2.0pt);
\fill [color=blue] (-1.32,-0.72) circle (2.0pt);
\fill [color=blue] ( 1.34,-0.68) circle (2.0pt);
\fill [color=blue] ( 0.7 ,-1.33) circle (2.0pt);
\draw [color=red][fill=white] (-1.42,-0.47) circle (2.0pt)[line width =1pt];
\draw [color=red][fill=white] ( 1.42,-0.47) circle (2.0pt)[line width =1pt];
\draw [color=red][fill=white] (-0.47,-1.42) circle (2.0pt)[line width =1pt];
\draw [color=red][fill=white] (-1.22,-0.88) circle (2.0pt)[line width =1pt];
\draw [color=red][fill=white] ( 0.47,-1.42) circle (2.0pt)[line width =1pt];
\draw [color=red][fill=white] ( 0.93,-1.18) circle (2.0pt)[line width =1pt];
\draw [color=red][fill=white] ( 1.23,-0.86) circle (2.0pt)[line width =1pt];
\end{scriptsize}
\draw (0,-1.8) node{$\mathrm{II}.\ \Surf_{\BD(H_2)}$};
\end{tikzpicture}
\ \ \ \ \
\begin{tikzpicture}[scale=1.3]
\fill [left color = cyan!50] (-1.42,-0.47) arc(198:-18:1.5);
\draw [cyan]    ( 0.00, 0.30) node[right]{$\bfS_{\BD(H_1)}^{\inc}$};
\fill [left color = green!50] (-1.42,-0.47) arc(198:342:1.5);
\draw [Green]    ( 0.00,-0.70) node[right]{$\bfS_{\BD(H_2)}^{\inc}$};
\draw [line width=1.5pt] (0,0) circle (1.5cm);
\draw [blue] (0,1.5)-- (-1.5 , 0   ) [line width =.5pt][dashed];
\draw [blue] (0,1.5)-- ( 1.5 , 0   ) [line width =.5pt][dashed];
\draw [blue] (0,1.5)-- ( 0   ,-1.5 ) [line width = 1pt][dashed];
\draw [blue] (0,1.5)-- ( 1.11, 1.01) [line width = 1pt];
\draw [red]  (-1.11, 1   )-- (-1.42,-0.47) [line width =.5pt][dashed];
\draw [red]  (-1.42,-0.47)-- ( 1.42,-0.47) [line width = 1pt];
\draw [red]  ( 0.5 , 1.42)-- ( 1.41, 0.51) [line width = 1pt];
\draw [red]  ( 1.41, 0.51)-- ( 1.42,-0.47) [line width =.5pt][dashed];
\draw [blue] (0,-1.5)-- ( 0   , 1.5 ) [line width = 1pt];
\draw [blue] (0,-1.5)-- ( 1.11,-1   ) [line width = 1pt];
\draw [blue] (0,-1.5)-- (-1.12,-1   ) [line width = 1pt];
\draw [blue] (0,-1.5)-- (-1.32,-0.72) [line width =.5pt][dashed];
\draw [blue] (0,-1.5)-- ( 1.34,-0.68) [line width =.5pt][dashed];
\draw [blue] (0,-1.5)-- ( 0.7 ,-1.33) [line width =.5pt][dashed] ;
\draw [red]  (-1.42,-0.47)-- (1.42,-0.47) [line width =1pt];
\draw [red] (-1.42,-0.47) to[out=   0, in=  45] (-1.22,-0.88)[line width =.5pt][dashed];
\draw [red] (-1.22,-0.88) to[out= -20, in= 135] (-0.47,-1.42)[line width = 1pt];
\draw [red] ( 0.47,-1.42) to[out=  80, in= 180] ( 0.93,-1.18)[line width =.5pt][dashed];
\draw [red] ( 0.93,-1.18) to[out=  90, in= 180] ( 1.23,-0.86)[line width =1pt];
\draw [red] ( 1.23,-0.86) to[out=  90, in= 180] ( 1.42,-0.47)[line width =1pt][dashed];
\begin{scriptsize}
\fill [color=blue] ( 0   , 1.5 ) circle (2.0pt);
\fill [color=blue] (-1.5 , 0   ) circle (2.0pt);
\fill [color=blue] ( 1.5 , 0   ) circle (2.0pt);
\fill [color=blue] ( 0   ,-1.5 ) circle (2.0pt);
\fill [color=blue] ( 1.11, 1.01) circle (2.0pt);
\draw [color=red][fill=white]  ( 1.41, 0.51) circle (2.0pt) [line width =1pt];
\draw [color=red][fill=white]  ( 1.42,-0.47) circle (2.0pt) [line width =1pt];
\draw [color=red][fill=white]  (-1.42,-0.47) circle (2.0pt) [line width =1pt];
\draw [color=red][fill=white]  (-1.11, 1   ) circle (2.0pt) [line width =1pt];
\draw [color=red][fill=white]  ( 0.5 , 1.42) circle (2.0pt) [line width =1pt];
\fill [color=blue] ( 0   ,-1.5 ) circle (2.0pt);
\fill [color=blue] ( 0   , 1.5 ) circle (2.0pt);
\fill [color=blue] ( 1.11,-1   ) circle (2.0pt);
\fill [color=blue] (-1.12,-1   ) circle (2.0pt);
\fill [color=blue] (-1.32,-0.72) circle (2.0pt);
\fill [color=blue] ( 1.34,-0.68) circle (2.0pt);
\fill [color=blue] ( 0.7 ,-1.33) circle (2.0pt);
\draw [color=red][fill=white]  (-1.42,-0.47) circle (2.0pt) [line width =1pt];
\draw [color=red][fill=white]  ( 1.42,-0.47) circle (2.0pt) [line width =1pt];
\draw [color=red][fill=white]  (-0.47,-1.42) circle (2.0pt) [line width =1pt];
\draw [color=red][fill=white]  (-1.22,-0.88) circle (2.0pt) [line width =1pt];
\draw [color=red][fill=white]  ( 0.47,-1.42) circle (2.0pt) [line width =1pt];
\draw [color=red][fill=white]  ( 0.93,-1.18) circle (2.0pt) [line width =1pt];
\draw [color=red][fill=white]  ( 1.23,-0.86) circle (2.0pt) [line width =1pt];
\end{scriptsize}
\draw (0,-1.8) node{$\mathrm{III}.\ \Surf_{\BD(A)}$};
\end{tikzpicture}
\label{fig:BD-alg}
\caption{The marked surfaces of $\BD(H_1)$, $\BD(H_2)$ and $\BD(A)$ given in Example \ref{exp:BD-alg}}
\end{figure}
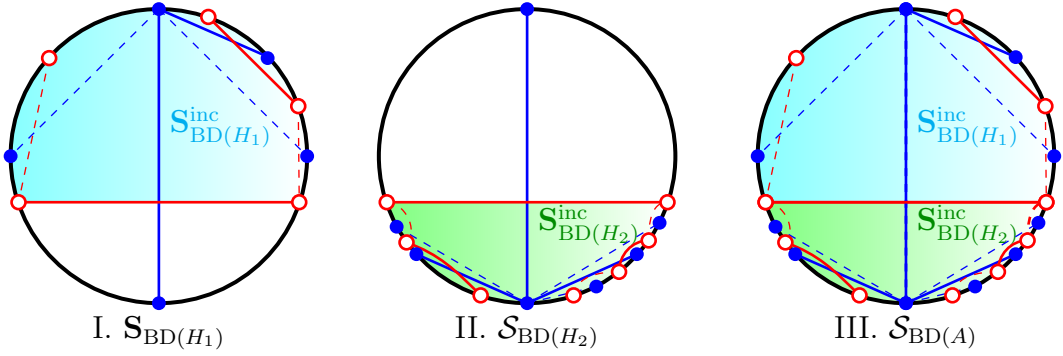
The marked surface $\bfS_{\BD(A)}$ can be obtained by gluing
the edge $\av^{-1}((1,1))^{\perp}$ of $\bfS_{\BD(H_1)}^{\inc}$
and the edge $\av^{-1}((2,2))^{\perp}$ of $\bfS_{\BD(H_2)}^{\inc}$, see \Pic \ref{fig:BD-alg}, III.
Then we obtain that the quiver of $\BD(A)$ is
\[\tiny\xymatrix@R=0.35cm@C=0.7cm{
(2,1)'\ar[r]& (2,1)\ar[r] & (2,2)'\ar[rd] & & (1,2)' \ar[r] & (1,2) \\
& & & (1,1)=(2,2) \ar[ru]\ar[rd] & &  \\
& & (1,1)' \ar[ru] & & (2,3)'\ar[r] & (2,3)
}\]
by using the correspondence between marked surfaces and gentle algebras in Theorem \ref{thm:corresp}.
\end{example}

Next, we provide an instance of a gentle algebra with infinite global dimension.

\begin{example} \rm \label{exp:gent gldim=infty}
Let $A=\kk\Q/\I$ be a gentle algebra whose bound quiver $(\Q,\I)$ and marked surface
are shown in the first picture of \Pic \ref{fig:gent gldim=infty}.
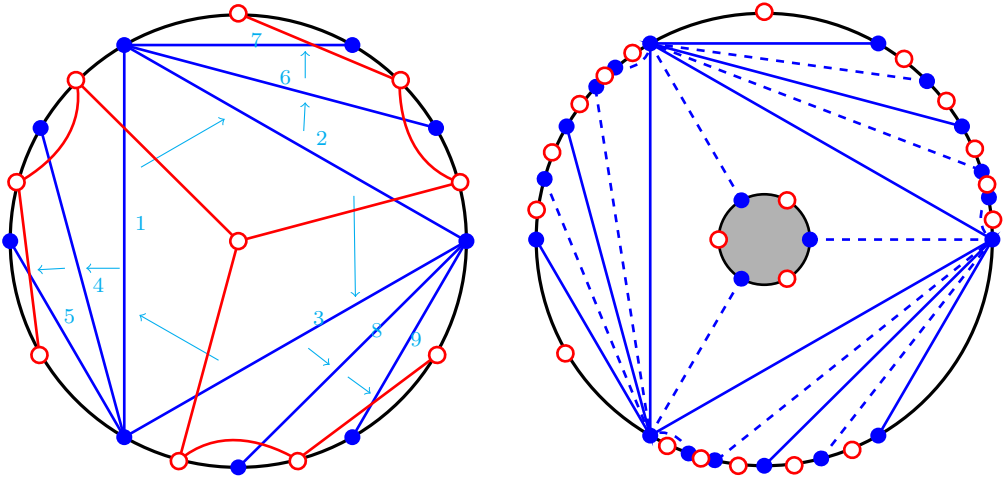
\begin{figure}[htbp]
\begin{center}
\begin{tikzpicture}[scale=2]
\draw [line width=1.2pt] (0,0) circle (1.5cm);
\draw [blue] (-0.75, 1.30)-- (-0.75,-1.30) [line width =1pt];
\draw [blue] (-0.75,-1.30)-- ( 1.50, 0.00) [line width =1pt];
\draw [blue] ( 1.50, 0.00)-- (-0.75, 1.30) [line width =1pt];
\draw [blue] (-0.75, 1.30)-- ( 0.75, 1.30) [line width =1pt];
\draw [blue] ( 1.50, 0.00)-- ( 0.75,-1.30) [line width =1pt];
\draw [blue] (-0.75,-1.30)-- (-1.50, 0.00) [line width =1pt];
\draw [blue] ( 1.50, 0.00)-- ( 0.00,-1.50) [line width =1pt];
\draw [blue] (-0.75,-1.30)-- (-1.30, 0.75) [line width =1pt];
\draw [blue] (-0.75, 1.30)-- ( 1.30, 0.75) [line width =1pt];
\draw [->,cyan] ( 0.76, 0.30) -- ( 0.77,-0.37);
\draw [->,cyan] (-0.64, 0.49) -- (-0.09, 0.81);
\draw [->,cyan] (-0.13,-0.79) -- (-0.65,-0.49);
\draw [->,cyan] ( 0.43, 0.73) -- ( 0.44, 0.92);
\draw [->,cyan] ( 0.44, 1.08) -- ( 0.44, 1.26);
\draw [->,cyan] ( 0.46,-0.71) -- ( 0.60,-0.82);
\draw [->,cyan] ( 0.72,-0.90) -- ( 0.87,-1.01);
\draw [->,cyan] (-0.78,-0.18) -- (-1.00,-0.18);
\draw [->,cyan] (-1.14,-0.18) -- (-1.32,-0.19);
\draw [cyan] (-0.64, 0.12) node {\tiny$1$};
\draw [cyan] ( 0.55, 0.68) node {\tiny$2$};
\draw [cyan] ( 0.53,-0.51) node {\tiny$3$};
\draw [cyan] (-0.92,-0.29) node {\tiny$4$};
\draw [cyan] (-1.11,-0.50) node {\tiny$5$};
\draw [cyan] ( 0.31, 1.09) node {\tiny$6$};
\draw [cyan] ( 0.12, 1.33) node {\tiny$7$};
\draw [cyan] ( 0.91,-0.59) node {\tiny$8$};
\draw [cyan] ( 1.17,-0.65) node {\tiny$9$};
\fill [blue] ( 1.50, 0.00) circle (1.5pt);
\fill [blue] ( 0.75, 1.30) circle (1.5pt);
\fill [blue] (-0.75, 1.30) circle (1.5pt);
\fill [blue] (-1.50, 0.00) circle (1.5pt);
\fill [blue] (-0.75,-1.30) circle (1.5pt);
\fill [blue] ( 0.75,-1.30) circle (1.5pt);
\fill [blue] ( 0.00,-1.50) circle (1.5pt);
\fill [blue] (-1.30, 0.75) circle (1.5pt);
\fill [blue] ( 1.30, 0.75) circle (1.5pt);
\foreach \x in {45,165,285}
\draw [red][rotate = \x]  ( 0.00, 0.00) -- ( 0.00, 1.51) [line width=1pt];
\foreach \x in {0,-120,-240}
\draw [red][rotate = \x]  ( 0.00, 1.51) -- ( 1.06, 1.06)
  to[out=-90,in=160] ( 1.45, 0.39) [line width=1pt];
\foreach \x in {0,45,75,120,165,195,240,285,315}
\draw [red][fill=white][rotate = \x]  ( 0.00, 1.51) circle (1.5pt) [line width=1pt];
\draw [red][fill=white][rotate =  0]  ( 0.00, 0.00) circle (1.5pt) [line width=1pt];
%
\end{tikzpicture}
\ \
\ \
\begin{tikzpicture}[scale=2]
\draw [line width=1.2pt] (0,0) circle (1.5cm);
\draw [blue] (-0.75, 1.30)-- (-0.75,-1.30) [line width =1pt];
\draw [blue] (-0.75,-1.30)-- ( 1.50, 0.00) [line width =1pt];
\draw [blue] ( 1.50, 0.00)-- (-0.75, 1.30) [line width =1pt];
\draw [blue] (-0.75, 1.30)-- ( 0.75, 1.30) [line width =1pt];
\draw [blue] ( 1.50, 0.00)-- ( 0.75,-1.30) [line width =1pt];
\draw [blue] (-0.75,-1.30)-- (-1.50, 0.00) [line width =1pt];
\draw [blue] ( 1.50, 0.00)-- ( 0.00,-1.50) [line width =1pt];
\draw [blue] (-0.75,-1.30)-- (-1.30, 0.75) [line width =1pt];
\draw [blue] (-0.75, 1.30)-- ( 1.30, 0.75) [line width =1pt];
\fill [blue] ( 1.50, 0.00) circle (1.5pt);
\fill [blue] ( 0.75, 1.30) circle (1.5pt);
\fill [blue] (-0.75, 1.30) circle (1.5pt);
\fill [blue] (-1.50, 0.00) circle (1.5pt);
\fill [blue] (-0.75,-1.30) circle (1.5pt);
\fill [blue] ( 0.75,-1.30) circle (1.5pt);
\fill [blue] ( 0.00,-1.50) circle (1.5pt);
\fill [blue] (-1.30, 0.75) circle (1.5pt);
\fill [blue] ( 1.30, 0.75) circle (1.5pt);
\draw [black][fill=black!30][line width=1pt] (0,0) circle (3mm);
\foreach \x in {0,120,240}
\fill [blue][rotate=\x] (-0.15,0.26) circle (1.5pt);
\foreach \x in {0,120,240}
\draw [blue][rotate=\x]
  (-0.15, 0.26) --
  (-0.75, 1.30) to[out=-90,in=-15]
  (-0.98, 1.14) [line width=1pt][dashed];
\foreach \x in {0,120,240}
\draw [blue][rotate=\x]
  ( 1.07, 1.05) --
  (-0.75, 1.30) --
  ( 1.45, 0.45) [line width=1pt][dashed];
\foreach \x in {0,120,240}
\fill [blue][rotate=\x]
  ( 1.07, 1.05) circle (1.5pt)
  ( 1.43, 0.45) circle (1.5pt)
  (-0.98, 1.14) circle (1.5pt);
\foreach \x in {  0, 35, 44, 53.5, 67.5, 82.5,
                120,155,164,173.5,187.5,202.5
                240,275,284,293.5,307.5,322.5}
\draw [red][fill=white][rotate = \x]  ( 0.00, 1.51) circle (1.5pt) [line width=1pt];
\foreach \x in { 60,180,300}
\draw [red][fill=white][rotate = \x]  ( 0.30, 0.00) circle (1.5pt) [line width=1pt];
\end{tikzpicture}
\caption{Gentle algebra given by this marked surface has an infinite global dimension}
\label{fig:gent gldim=infty}
\end{center}
\end{figure}
Then $\gldim A = \infty$, and $\bfS_{\BD(A)}$ is shown in the second picture of \Pic \ref{fig:gent gldim=infty}
(we ignore the $\rbullet$-FFAS of $\bfS_{\BD(A)}$ for brevity).
\end{example}

\subsection{Marked surfaces of CM-Auslander algebras}

We can draw the marked surfaces of the CM-Auslander algebras of gentle algebras by using Theorem \ref{thm:CL}.
To do this, we recall some terms about forbidden path.

A {\defines forbidden path} is a path $a_1a_2\cdots a_l$ on the bound quiver $(\Q,\I)$ of gentle algebra
such that $a_ia_{i+1} \in \I$ holds for all $1\=< i\=< l-1$.
Furthermore, we call it a {\defines forbidden thread} if it is maximal,
that is, for any $\alpha$ with $\target(\alpha)=\source(a_1)$ we have $\alpha a_1\notin\I$
and for any $\beta$ with $\source(\beta)=\target(a_l)$ we have $a_l\beta\notin\I$.
{\defines Permitted path} is the dual of a forbidden path, that is,
it is a path $a_1a_2\cdots a_l$ which does not in $\I$.
We can define {\defines permitted thread} in a dual way.

\begin{remark} \rm
Permitted threads are introduced by Wald and Waschb\"{u}sch, which are originally said to be {\defines directed V$-$sequences} in \cite{WW1985}.
In \cite{BR1987}, V$-$sequences are used to describe the indecomposable module over string algebra, and they are called strings.
In \cite{AG2008}, permitted threads and forbidden threads are used to describe an invariant of the derived category of gentle algebra.
Today, this invariant is said to be the Avella-Alaminos$-$Gie\ss\ invariant (= AG-invariant for short) of gentle algebra.
\end{remark}

Let $\bEP(\bfS)$ be the set of all $\bbullet$-elementary polygons of a marked surface $\bfS$.
$\bEP(\bfS)$ contains two classes of $\bbullet$-elementary polygons:
the class $\bEP^{\infty}(\bfS)$ of all $\infty$-elementary polygons;
and the class $\bEP^{\pounds}(\bfS) = \bEP(\bfS)\backslash \bEP^{\infty}(\bfS)$ of other $\bbullet$-elementary polygons.
Then, by using the correspondence given in Theorem \ref{thm:corresp}, there is a bijection
\[ F^{\pounds}(A) \to \bEP^{\pounds}(\bfS) \]
from the set $F^{\pounds}(A)$ of all forbidden threads to $\bEP^{\pounds}(\bfS)$, and there is a bijection
\[ F^{\infty}(A) \to \bEP^{\infty}(\bfS)\]
form the set $F^{\infty}(A)$ of all forbidden cycles to $\bEP^{\infty}(\bfS)$.
Here, $A \cong A(\bfS)$. That is, we have the bijection
\begin{align}\label{formula:forb-EP}
  F(A) := F^{\pounds}(A)\cup F^{\infty}(A) \mathop{\rightleftarrows}\limits^{F^{-1}}_{F} \bEP(\bfS_A)
\end{align}

\begin{construction} \label{const:CMA} \rm
Let $\bfS$ be a marked surface ignoring $\rbullet$-FFAS,
then every $\rbullet$-marked point is rewritten as a boundary component of $\Surf$.
In this case, $\bfS=(\Surf, \Marked, \bAS)$.

For each $\infty-$elementary polygon (assume the number of edges of it is $\ell \>= 1$,
see the first picture of \Pic \ref{fig:infty poly II}),
\begin{figure}[htbp]
\centering
\begin{tikzpicture}[scale=1.4]
\foreach \x in {0, 60, 120, 180, 300}
\draw[blue][rotate = \x] ( 1.73,-1.00) -- ( 1.73, 1.00) [line width=1pt];
\draw[blue][rotate =240] ( 1.73,-1.00) -- ( 1.73, 1.00) [line width=1pt][dotted];
\draw[blue][rotate = 60] ( 1.73*1.2, 1.00*1.2) node{$v_1$};
\draw[blue][rotate =120] ( 1.73*1.2, 1.00*1.2) node{$v_2$};
\draw[blue][rotate =180] ( 1.73*1.2, 1.00*1.2) node{$v_3$};
\draw[blue][rotate =240] ( 1.73*1.2, 1.00*1.2) node{$v_{\ell-2}$};
\draw[blue][rotate =300] ( 1.73*1.2, 1.00*1.2) node{$v_{\ell-1}$};
\draw[blue][rotate =  0] ( 1.73*1.2, 1.00*1.2) node{$v_{\ell}$};
\foreach \x in {0, 60, ..., 300}
\fill[blue][rotate = \x] ( 1.73,-1.00) circle(1mm);
\draw[black][fill= black!30][line width=1pt] (0,0) circle(5mm);
\draw[black] (0,-0.5) node[above]{$\bcomp$};
\draw (1.732-0.7,0) node{$\Poly$};
\foreach \x in{60,120,180,240,300,360}
\draw[cyan][->][line width=1pt][rotate=-30+\x] (0.2,-1.73) arc(180:62:0.8);
\draw[cyan][rotate= -60] (0.8, 0.5) node{$a_{\ell-1}$};
\draw[cyan][rotate=   0] (0.8, 0.5) node{$a_{\ell}$};
\draw[cyan][rotate=  60] (0.8, 0.5) node{$a_{1}$};
\draw[cyan][rotate= 120] (0.8, 0.5) node{$a_{2}$};
\draw[cyan][rotate= 180] (0.8, 0.5) node{$a_{3}$};
\draw[cyan][rotate= 240] (0.8, 0.5) node{$a_{\ell-2}$};
\draw[cyan][rotate= 205] (0.8, 0.5) node{$\ddots$};
\draw[rotate=   0] (1.732+0.3,0) node{$\bluea{\ell-1}{\Poly}$};
\draw[rotate=  60] (1.732+0.3,0) node{$\bluea{\ell}{\Poly}$};
\draw[rotate= 120] (1.732+0.3,0) node{$\bluea{1}{\Poly}$};
\draw[rotate= 180] (1.732+0.3,0) node{$\bluea{2}{\Poly}$};
\draw[rotate= 300] (1.732+0.3,0) node{$\bluea{\ell-2}{\Poly}$};
\end{tikzpicture}
\begin{tikzpicture}[scale=1.4]
\draw[white] (0,-2.5) -- (0,2.5);
\draw (0,0) node{$\Longrightarrow$};
\end{tikzpicture}
\begin{tikzpicture}[scale=1.4]
\foreach \x in {0, 60, 120, 180, 300}
\draw[blue][rotate = \x] ( 1.73,-1.00) -- ( 1.73, 1.00) [line width=1pt];
\draw[blue][rotate =240] ( 1.73,-1.00) -- ( 1.73, 1.00) [line width=1pt][dotted];
\draw[blue][rotate = 60] ( 1.73*1.2, 1.00*1.2) node{$v_1$};
\draw[blue][rotate =120] ( 1.73*1.2, 1.00*1.2) node{$v_2$};
\draw[blue][rotate =180] ( 1.73*1.2, 1.00*1.2) node{$v_3$};
\draw[blue][rotate =240] ( 1.73*1.2, 1.00*1.2) node{$v_{\ell-2}$};
\draw[blue][rotate =300] ( 1.73*1.2, 1.00*1.2) node{$v_{\ell-1}$};
\draw[blue][rotate =  0] ( 1.73*1.2, 1.00*1.2) node{$v_{\ell}$};
\foreach \x in {0, 60, ..., 300}
\fill[blue][rotate = \x] ( 1.73,-1.00) circle(1mm);
\draw[black][fill= black!30][line width=1pt] (0,0) circle(5mm);
\foreach \x in {0, 60, ..., 300}
\fill[blue][rotate = \x] ( 0.00, 0.50) circle(1mm);
\foreach \x in {0, 60, ..., 300}
\draw[blue][line width=1pt][rotate = \x] ( 0.00, 0.50) -- ( 0.00,2.00);
\draw[blue] (-0.7,-1) node{$\ddots$};
\draw[blue][rotate=   0] ( 0.20, 0.60) node{\tiny$w_1$};
\draw[blue][rotate=  60] ( 0.20, 0.60) node{\tiny$w_2$};
\draw[blue][rotate= 120] ( 0.20, 0.60) node{\tiny$w_3$};
\draw[blue][rotate= 180] ( 0.30, 0.60) node{\tiny$w_{\ell-2}$};
\draw[blue][rotate= 240] ( 0.20, 0.70) node{\tiny$w_{\ell-1}$};
\draw[blue][rotate= 300] ( 0.20, 0.60) node{\tiny$w_{\ell}$};
\draw[rotate=   0] (1.732+0.3,0) node{$\bluea{\ell-1}{\Poly}$};
\draw[rotate=  60] (1.732+0.3,0) node{$\bluea{\ell}{\Poly}$};
\draw[rotate= 120] (1.732+0.3,0) node{$\bluea{1}{\Poly}$};
\draw[rotate= 180] (1.732+0.3,0) node{$\bluea{2}{\Poly}$};
\draw[rotate= 300] (1.732+0.3,0) node{$\bluea{\ell-2}{\Poly}$};
\draw[rotate=   0] (0.3,1.2) node{$\bluea{1}{}$};
\draw[rotate=  60] (0.3,1.2) node{$\bluea{2}{}$};
\draw[rotate= 120] (0.3,1.2) node{$\bluea{3}{}$};
\draw[rotate= 185] (0.3,1.2) node{$\bluea{\ell-2}{}$};
\draw[rotate= 240] (0.3,1.2) node{$\bluea{\ell-1}{}$};
\draw[rotate= 300] (0.3,1.2) node{$\bluea{\ell}{}$};
\end{tikzpicture}
\caption{$\infty-$elementary polygon}
\label{fig:infty poly II}
\end{figure}
we add $\ell$ $\rbullet$-marked points $w_1$, $w_2$, $\ldots$, $w_{\ell}$ to $\bcomp$ in a counterclockwise direction,
and add $\ell$ $\rbullet$-arcs $\bluea{1}{}$, $\bluea{2}{}$, $\ldots$, $\bluea{\ell}{}$ such that
the ending points of $\bluea{i}{}$ are $v_i$ and $w_i$, see the second picture of \Pic \ref{fig:infty poly II}.
Then we obtain a new marked surface which is denoted by $\bfS^{\CMA} = (\Surf^{\CMA}, \Marked^{\CMA}, \AS^{\CMA})$.
Obviously, if $\bfS$ does not have any $\infty$-elementary polygon, then $\bfS^{\CMA}=\bfS$.
\end{construction}

The following proposition shows that the marked surface of the CM-Auslander algebra of a gentle algebra $A$ is homotopic to $\bfS_A^{\CMA}$.

\begin{proposition} \label{prop:CMAsurface} \
\begin{itemize}
  \item[\rm(1)] Let $A$ be a gentle algebra and $\bfS_A$ be its marked surface.
    Then $\bfS_{A^{\CMA}}$ is homotopic to $\bfS_A^{\CMA}$.
  \item[\rm(2)] Let $\bfS$ be a marked surface and $A(\bfS)$ be its gentle algebra.
    Then $A(\bfS)^{\CMA} \cong A(\bfS^{\CMA})$.
\end{itemize}
\end{proposition}

\begin{proof}
We only prove (1), and the statement (2) is a direct corollary of the statement (1) and Theorem \ref{thm:corresp}.

For any $\Poly\in\bEP(\bfS_A)$, we have two cases as follows.
\begin{itemize}
  \item[(A)] If $\Poly \in \bEP^{\infty}(\bfS_A)$, that is, $\Poly$ is of the form
    shown in the first picture of \Pic \ref{fig:infty poly II},
    then it corresponds to the polygon $\Poly^{\CMA}$ of $\bfS_A^{\CMA}$ which is of the form
    shown in the second picture of \Pic \ref{fig:infty poly II}.
    Here, notice that $\Poly^{\CMA}$ is a union of $\ell$ $\bbullet$-elementary polygons
    lying in $\bEP^{\pounds}(\bfS_A^{\CMA})$.
  \item[(B)] If $\Poly \in \bEP^{\pounds}(\bfS_A)$, then the $\bbullet$-elementary polygon $\Poly^{\CMA}$ in $\bEP^{\pounds}(\bfS_A^{\CMA})$ coincides with $\Poly$.
\end{itemize}

By formula (\ref{formula:forb-EP}), the forbidden path $F(\Poly)$ has two cases (a) and (b) which correspond to (A) and (B), respectively.
\begin{itemize}
  \item[(a)] In Case (A), $F(\Poly)$ is a forbidden cycle $a_1a_2\cdots a_{\ell}$.
    By Theorem \ref{thm:CL}, each arrow $a_{i}$ in $F(\Poly)$ splits two arrows
    \[ a_{i}^{-}: \av(\bluea{\overline{i-1}}{\Poly}) \to a_iA
    \text{ and }  a_{i}^{+}: a_iA \to \av(\bluea{i}{\Poly}), \]
    where $\overline{x} =
    \begin{cases}
      x \bmod \ell, & x \notin \ell\mathbb{Z};  \\
      \ell, & x \in \ell\mathbb{Z}.
    \end{cases}$
    That is, $F(\Poly)$ corresponds to the oriented cycle
    \[ a_1^{-}a_1^{+}a_2^{-}a_2^{+}\cdots a_{\ell}^{-}a_{\ell}^{+} \]
    in $(\Q^{\CMA},\I^{\CMA})$. Here, $a_1^{+}a_2^{-}, \ldots, a_{\ell-1}^{+}a_{\ell}^{-}, a_{\ell}^{+}a_1^{-} \in \I^{\CMA}$.

  \item[(b)] In Case (B), $F(\Poly) \in F(A)$ is a forbidden thread.
    By Theorem \ref{thm:CL}, the forbidden thread in $F(A^{\CMA})$ corresponding to $F(\Poly)$
    coincides with $F(\Poly)$.
\end{itemize}

The changes given in (A), (B), (a), and (b) show that $\bfS_{A^{\CMA}}$ is homotopic to $\bfS_A^{\CMA}$.
\end{proof}

\begin{example} \rm
Let $A=\kk\Q/\I$ be a gentle algebra as in Example \ref{exp:gent gldim=infty}.
Its marked surface is shown in \Pic \ref{fig:gent gldim=infty}.
Then the marked surface of the CM-Auslander algebra $A^{\CMA}$ of $A$ is shown in \Pic \ref{fig:CMA}.
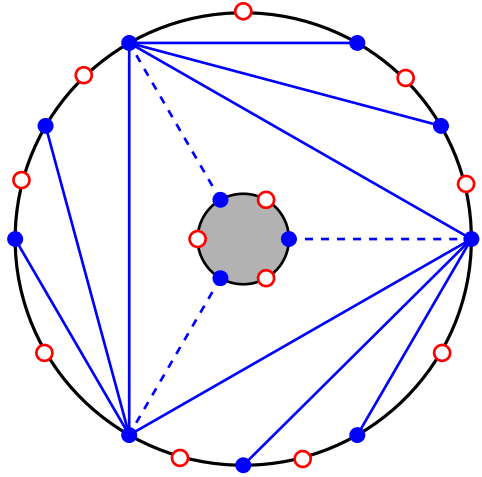
\begin{figure}[htbp]
\centering
\begin{tikzpicture}[scale=2]
\draw [line width=1.2pt] (0,0) circle (1.5cm);
\draw [blue] (-0.75, 1.30)-- (-0.75,-1.30) [line width =1pt];
\draw [blue] (-0.75,-1.30)-- ( 1.50, 0.00) [line width =1pt];
\draw [blue] ( 1.50, 0.00)-- (-0.75, 1.30) [line width =1pt];
\draw [blue] (-0.75, 1.30)-- ( 0.75, 1.30) [line width =1pt];
\draw [blue] ( 1.50, 0.00)-- ( 0.75,-1.30) [line width =1pt];
\draw [blue] (-0.75,-1.30)-- (-1.50, 0.00) [line width =1pt];
\draw [blue] ( 1.50, 0.00)-- ( 0.00,-1.50) [line width =1pt];
\draw [blue] (-0.75,-1.30)-- (-1.30, 0.75) [line width =1pt];
\draw [blue] (-0.75, 1.30)-- ( 1.30, 0.75) [line width =1pt];
\fill [blue] ( 1.50, 0.00) circle (1.5pt);
\fill [blue] ( 0.75, 1.30) circle (1.5pt);
\fill [blue] (-0.75, 1.30) circle (1.5pt);
\fill [blue] (-1.50, 0.00) circle (1.5pt);
\fill [blue] (-0.75,-1.30) circle (1.5pt);
\fill [blue] ( 0.75,-1.30) circle (1.5pt);
\fill [blue] ( 0.00,-1.50) circle (1.5pt);
\fill [blue] (-1.30, 0.75) circle (1.5pt);
\fill [blue] ( 1.30, 0.75) circle (1.5pt);
\draw [black][fill=black!30][line width=1pt] (0,0) circle (3mm);
\foreach \x in {0,120,240}
\fill [blue][rotate=\x] (-0.15,0.26) circle (1.5pt);
\foreach \x in {0,120,240}
\draw [blue][rotate=\x]
  (-0.15, 0.26) --
  (-0.75, 1.30) [line width=1pt][dashed];
\foreach \x in {  0, 44, 75,
                120,164,195,
                240,284,315}
\draw [red][fill=white][rotate = \x]  ( 0.00, 1.51) circle (1.5pt) [line width=1pt];
\foreach \x in { 60,180,300}
\draw [red][fill=white][rotate = \x]  ( 0.30, 0.00) circle (1.5pt) [line width=1pt];
\end{tikzpicture}
\caption{The CM-Auslander algebra of the gentle algebra given in Example \ref{exp:gent gldim=infty}}
\label{fig:CMA}
\end{figure}

\end{example}

\section{Representation types and derived representation types of gentle algebras} \label{Sect:repr-type}

A finite-dimensional $\kk$-algebra $A$ is said to be
\begin{itemize}
  \item[(1)] {\defines representation-finite} if $\sharp\ind(\modcat A)<\infty$,
    where $\ind(\modcat A)$ is the set of all indecomposable modules up to isomorphism,
    and ``$\sharp$'' is a map sending each set $S$ to the number of all elements of $S$;
  \item[(2)] {\defines representation-infinite} if $\sharp\ind(\modcat A)=\infty$;
  \item[(3)] {\defines derived discrete} if for each cohomology dimension vector, there admit only finitely many objects in $\Dcat^b(A)$ up to isomorphism;
  \item[(4)] {\defines strongly derived unbounded} if for each cohomology dimension vector, there admit infinitely many objects in $\Dcat^b(A)$ up to isomorphism and shift.
\end{itemize}

In this section, we will study the representation type of gentle algebra.

\subsection{Representation types of gentle algebras} \label{subsect:repr type}

In \cite{BR1987}, the authors introduce strings and bands and describe the finitely generated module category of gentle algebra.
Next, we recall the definitions of strings and bands.

For the bound quiver $(\Q,\I)$ of any gentle algebra $A=\kk\Q/\I$,
we define $\Q_1^{-1} = \{a^{-1} \mid a\in \Q_1\}$ is a set such that there is a bijection
\[  \Q_1 \to \Q_1^{-1} \]
sending each arrow $a\in \Q_1$ to the {\defines formal inverse} $a^{-1}$ of $a$
Then the {\defines formal inverse} of a path $p=a_1a_2\cdots a_{\ell}$ is
\[ p^{-1} = a_{\ell}^{-1}\cdots a_2^{-1}a_1^{-1}. \]
For any path or any formal inverse of a path $\Path$, we naturally define $(\Path^{-1})^{-1}=p$.
In particular, for each path $\e_v$ of length zero corresponding to $v\in \Q_0$,
we define $\e_v^{-1}=\e$. Thus, we have $\Q_0^{-1}=\Q_0$.
For any formal inverse $p^{-1}$, we define $\source(p^{-1}) = \target(p)$ and $\target(p^{-1}) = \source(p)$.
Next, we give the definitions of string and band.

\begin{definition} \rm \label{def:str band}
A {\defines string} is a sequence $\str = p_1p_2\cdots p_m$ such that the following conditions hold.
\begin{itemize}
  \item[(1)] Each $p_i$ ($1 \=< i\=< m$) is either a path or the formal inverse of a path. Furthermore,
    \begin{itemize}
      \item if $p_i$ is a path, then it is a permitted path;
      \item if $p_i$ is the formal inverse of a path, then $p_i^{-1}$ is a permitted path.
    \end{itemize}

  \item[(2)] For $p_i$ and $p_{i+1}$, if $p_i$ is a path, then $p_{i+1}$ is the formal inverse of a path such that $\target(p_i) = \target(p_{i+1}^{-1})$;
    if $p_{i+1}$ is a path, then $p_i$ is the formal inverse of a path such that $\source(p_i^{-1}) = \source(p_{i+1})$.

  \item[(3)] If $p_i=a_{i,1}a_{i,2}\cdots a_{i,l_i}$ and $p_{i+1}=a_{i+1,l_{i+1}}^{-1} \cdots a_{i+1, 2}^{-1} a_{i+1,1}^{-1}$, then $a_{i,l_i}\ne a_{i+1,l_{i+1}}$;
      if $p_i=a_{i,l_i}^{-1}\cdots a_{i,2}^{-1}a_{i,1}^{-1}$ and $p_{i+1} = a_{i+1,1}a_{i+1,2}\cdots a_{i+1,l_{i+1}}$,
      then $a_{i,1} \ne a_{i+1,1}$.
\end{itemize}
Two strings $\str$ and $\str'$ are said to be {\defines equivalent}, write it as $\str\simeq \str'$, if one of $\str=\str'$ and $\str=(\str')^{-1}$ holds.

A {\defines band} is a string $\band = p_1p_2\cdots p_{m}$ with $\target(p_{m})=\source(p_1)$ satisfying the following conditions.
\begin{itemize}
  \item[(1)] It is not a non-trivial power of any string.
  \item[(2)] The power $\band^2$ is a string.
\end{itemize}
Two bands $\band$ and $\band'$ are said to be {\defines equivalent}, write it as $\band\simeq \band'$, if one of $\band[t]=\band'$ and $\band[t]=(\band')^{-1}$ holds,
where $t$ is an integer with $0\=< t \=< m$ and $\band[t] = p_{1+t}p_{2+t}\cdots p_{m}p_{1}\cdots p_{t}$.
\end{definition}

By the works of Butler and Ringel in \cite{BR1987}, we have the following description of the module category of gentle algebra.

\begin{theorem}[{Butler$-$Ringel Correspondence}] \label{thm:BR}
For any gentle algebra $A=\kk\Q/\I$, there exists the following bijection
\[ \M : \Str(A) \cup (\Band(A)\times \mathscr{J}) \to \ind(\modcat A), \]
where $\Str(A)$ is the set of all equivalence classes of strings,
$\Band(A)$ is the set of all equivalence classes of bands, and
$\mathscr{J}$ is the set of all Jordan blocks with non-zero eigenvalues.
\end{theorem}

Now, we study the relation between representation types of gentle algebra $A$ and that of BD-gentle algebra $\BD(A)$.
First, we have the following lemma which admits that the structure of the bound quiver of $A$ is similar to that of the bound quiver of $\BD(A)$.

\begin{lemma} \label{lemm:BDband}
The bound quiver of a gentle algebra $A$ has a band if and only if that of $\BD(A)$ has a band.
\end{lemma}

\begin{proof}
Assume $A=\kk\Q/\I$ is gentle such that $(\Q,\I)$ has a band,
say $\band=a_1a_2\cdots a_n$, where all $a_i$ lie in $\Q_1^{\pm 1}$,
and assume $\BD(A)=\kk\Q^{\BD}/\I^{\BD}$.
Obviously, $\band$ has at least one source $v$ and at least one sink $w$.
Otherwise, $\band$ is an oriented cycle without any relation,
and then $A$ is an infinite-dimensional algebra, a contradiction.
Thus, $\band$ is equivalent to the following form
\[  p_1p_2\cdots p_{2m}, \]
where $p_1$, $p_2^{-1}$, $p_3$, $p_4^{-1}$, $\cdots$, $p_{2m-1}$,  $p_{2m}^{-1}$ are permitted paths.
By Definition \ref{def:BD-gentle}, each $p_i$ ($i\in \{1,3,\cdots,2m-1\}$) and each $p_j^{-1}$ ($j\in\{2,4,\cdots,2m\}$) are subpaths of some permitted threads $H_i$ and $H_j$, respectively.
Notice that we do not deny the possibility of $H_i=H_j$.
Thus, by Theorem \ref{thm:BD(gent)}, the permitted path
\[p_k^{(-1)^{k+1}} =\ v_{k,1} \To{a_{k,1}}{} v_{k,2} \To{a_{k,2}}{} \cdots \To{a_{k,l_k}}{} v_{k,l_k+1}\]
in $(\Q,\I)$ changes to the permitted path
\[  q_k^{(-1)^{k+1}} = \
    v_{k,1}' \To{a_{k,1}'}{} v_{k,1} \To{a_{k,1}''}{}
    v_{k,2}' \To{a_{k,2}'}{} v_{k,2} \To{a_{k,2}''}{} \cdots
             \To{a_{k,l_k}''}{} v_{k,l_k+1}' \To{a_{k,l_k+1}'}{} v_{k,l_k+1} \]
which is a subpath of $\BD(H_k)$. Here, $1\=< k\=< 2m$.
Then $(\Q^{\BD},\I^{\BD})$ contains the band
\begin{align}\label{formula:BDband}
   q_1q_2\cdots q_{2m},
\end{align}
where $q_1$, $q_2^{-1}$, $q_3$, $q_4^{-1}$, $\cdots$, $q_{2m-1}$,  $q_{2m}^{-1}$ are permitted paths.

On the other hand, assume that $(\Q^{\BD}, \I^{\BD})$ has a band,
then it is of the form shown in (\ref{formula:BDband}),
and the number of arrows of all $q_k^{(-1)^{k+1}}$ are odd.
Assume that \[ q_k = \ w_{k,1} \To{b_{k,1}}{} w_{k,2} \To{b_{k,2}}{}
  w_{k,3} \To{b_{k,3}}{} w_{k,4} \To{b_{k,4}}{} \cdots
  \To{b_{k,2m-2}}{} w_{k,2m-1} \To{b_{k,2m-1}}{} w_{k,2m} \]
is a subpath of $H_k^{\BD}$, where $H_k^{\BD}$ is a permitted thread on $(\Q^{\BD},\I^{\BD})$.
Then there is a unique permitted path
\[ p_k = \ w_{k,2} \To{}{} w_{k,4} \To{}{} \cdots \To{}{} w_{k,2m} \]
on $(\Q,\I)$ such that it is a subpath of some permitted thread $H_k$ on $(\Q,\I)$,
and $\BD(H_k) = H_k^{\BD}$. Then we obtain a band $p_1p_2\cdots p_{2m}$ on $(\Q,\I)$ as required.
\end{proof}

\begin{remark} \label{rmk:BDband} \rm
We provide another proof of Lemma \ref{lemm:BDband} by using marked surface in this remark.
To do this, we recall the definition of a permissible curve $c$ which is a sequence
\[ c = (c_{i,i+1})_{0 \=< i \=< m(c)} = c_{0,1} c_{1,2} \cdots c_{m(c), m(c)+1} \ (m(c)\in\NN) \]
of some segments given by one of the pictures given in \Pic \ref{fig:arc segment I}
(these segments are called {\defines $\bAS$-arc segments})
\begin{figure}[htbp]
\definecolor{ffqqqq}{rgb}{1,0,0}
\definecolor{bluearc}{rgb}{0,0,1}
\begin{tikzpicture}[scale=1.25]
\draw[black] (-0.5,2)--( 0.5,2) [line width=1pt];
\draw[bluearc] ( 0, 2)--(-1, 0) [line width=1pt];
\draw[bluearc] ( 0, 2)--( 1, 0) [line width=1pt];
\fill[bluearc] ( 0, 2) circle (0.1cm);
\draw[orange][line width=1pt] (-0.5, 1) to[out=-45, in=-135] ( 0.5, 1);
\draw (0,-0.5) node{Case (1)};
\end{tikzpicture}
\ \
\begin{tikzpicture}[scale=1.25]
\draw[black] (0,2) to[out=180,in=90] (-2,0) [line width=1pt];
\draw[bluearc] ( 0, 2) to[out=-90,in=0] (-2, 0) [line width=1pt];
\fill[bluearc] ( 0, 2) circle (0.1cm);
\fill[bluearc] (-2, 0) circle (0.1cm);
\draw[red][fill=white] (-1.41, 1.41) circle (0.1cm) [line width=1pt];
\draw[orange] (-1.41, 1.41) -- (-0.57, 0.57) [line width=1pt];
\draw (-1,-0.5) node{Case (2)};
\end{tikzpicture}
\caption{$\bAS$-arc segments}
\label{fig:arc segment I}
\end{figure}
such that:
\begin{itemize}
  \item[(1)] $c(t_i) = c_{(i,i+1)}(t_i)$ and $c(t_{m(c)+1})=c_{(m(c),m(c)+1)}(t_{m(c)+1})$,
    where $0=t_0 < t_1 < \cdots < t_{m(c)} < t_{m(c)+1}=1$
    (the curve $c$ and each $\bAS$-arc segment $c_{(i,i+1)}$ are seen as function
    $c:[0,1] \to \Surf$ and $c_{(i,i+1)}: [t_i, t_{i+1}] \to \Surf$, respectively);
  \item[(2)] two adjacent $\bAS$-arc segment $c_{(i,i+1)}$ and $c_{(i+1,i+2)}$ are different,
    i.e., $c_{(i,i+1)}\ne c_{(i+1,i+2)}$ and $c_{(i,i+1)}\ne c_{(i+1,i+2)}^{-1}$;
  \item[(3)] $c(0), c(1) \in \Marked$ are $\rbullet$-marked points,
    or $c(0)=c(1)$ lies in $\Surf^{\circ}$.
\end{itemize}

Next, we prove Lemma \ref{lemm:BDband}.
Clearly, a gentle algebra $A$ is representation-infinite if and only if there is a permissible curve without endpoint in $\bfS_A$, see the thickest curve in the first picture of \Pic \ref{fig:curve-band}.
In this case, the permissible curve without endpoint corresponds to a band on the bound quiver of $A$.
\begin{figure}[htbp]
\definecolor{ffqqqq}{rgb}{1,0,0}
\definecolor{qqwuqq}{rgb}{0,0,1}
\begin{tikzpicture} [scale=1.5]
\draw    [line width=1.5pt] (0,0) circle (2cm)[dotted];
\filldraw[color = black!25] (0,0) circle (1cm);
\draw    [line width=1.5pt] (0,0) circle (1cm)[dotted];
\foreach \x in{0,90,180,270}
\draw [white][line width=12pt][rotate = \x](0,0) -- (1.2,0);
\fill [qqwuqq!50] ( 2.  , 0.  ) circle (\pointsizeII);
\fill [qqwuqq] ( 1.  , 1.73) circle (\pointsizeII);
\fill [qqwuqq] ( 1.73, 1.  ) circle (\pointsizeII);
\fill [qqwuqq] ( 0.  , 2.  ) circle (\pointsizeII);
\fill [qqwuqq] (-1.  , 1.73) circle (\pointsizeII);
\fill [qqwuqq] (-1.73, 1.  ) circle (\pointsizeII);
\fill [qqwuqq] (-2.  , 0.  ) circle (\pointsizeII);
\fill [qqwuqq] (-1.73,-1.  ) circle (\pointsizeII);
\fill [qqwuqq] (-1.  ,-1.73) circle (\pointsizeII);
\fill [qqwuqq] ( 0.  ,-2.  ) circle (\pointsizeII);
\fill [qqwuqq] ( 1.  ,-1.73) circle (\pointsizeII);
\fill [qqwuqq] ( 1.73,-1.  ) circle (\pointsizeII);
\fill [qqwuqq] (0,1) circle (\pointsizeII)  [opacity=0];
\fill [qqwuqq] (-1.41*0.5,1.41*0.5) circle (\pointsizeII);
\fill [qqwuqq] (-1,0) circle (\pointsizeII) [opacity=0];
\fill [qqwuqq] (-1.41*0.5,-1.41*0.5) circle (\pointsizeII);
\fill [qqwuqq] (0,-1) circle (\pointsizeII) [opacity=0];
\fill [qqwuqq] (1.41*0.5,-1.41*0.5) circle (\pointsizeII);
\fill [qqwuqq] (1,0) circle (\pointsizeII)  [opacity=0];
\fill [qqwuqq] (1.41*0.5,1.41*0.5) circle (\pointsizeII);
\draw[qqwuqq][line width=1pt][dotted][opacity=0.5] (0,2)--(0,1);
\draw[qqwuqq][line width=1pt] (0,2)--(-1.41*0.5,1.41*0.5);
\draw[qqwuqq][line width=1pt] (0,2)--(1.41*0.5,1.41*0.5);
\draw[qqwuqq][line width=1pt][dotted][opacity=0.5] (-2,0)--(-1,0);
\draw[qqwuqq][line width=1pt] (-2,0)--(-1.41*0.5,1.41*0.5);
\draw[qqwuqq][line width=1pt] (-2,0)--(-1.41*0.5,-1.41*0.5);
\draw[qqwuqq][line width=1pt][dotted][opacity=0.5] (0,-2)--(0,-1);
\draw[qqwuqq][line width=1pt] (0,-2)--(-1.41*0.5,-1.41*0.5);
\draw[qqwuqq][line width=1pt] (0,-2)--(1.41*0.5,-1.41*0.5);
\draw[qqwuqq][line width=1pt][dotted][opacity=0.5] (2,0)--(1,0);
\draw[qqwuqq][line width=1pt][dotted][opacity=0.5] (2,0)--(1.41*0.5,1.41*0.5);
\draw[qqwuqq][line width=1pt][dotted][opacity=0.5] (2,0)--(1.41*0.5,-1.41*0.5);
\draw[qqwuqq][line width=1pt][dotted][opacity=0.5] (1.41*0.5,1.41*0.5) -- (1.41,1.41);
\draw[qqwuqq][line width=1pt] (1.41*0.5,1.41*0.5) -- (1.73,1);
\draw[qqwuqq][line width=1pt] (1.41*0.5,1.41*0.5) -- (1,1.73);
\draw[qqwuqq][line width=1pt][dotted][opacity=0.5] (-1.41*0.5,1.41*0.5) -- (-1.41,1.41);
\draw[qqwuqq][line width=1pt] (-1.41*0.5,1.41*0.5) -- (-1.73,1);
\draw[qqwuqq][line width=1pt] (-1.41*0.5,1.41*0.5) -- (-1,1.73);
\draw[qqwuqq][line width=1pt][dotted][opacity=0.5] (-1.41*0.5,-1.41*0.5) -- (-1.41,-1.41);
\draw[qqwuqq][line width=1pt] (-1.41*0.5,-1.41*0.5) -- (-1.73,-1);
\draw[qqwuqq][line width=1pt] (-1.41*0.5,-1.41*0.5) -- (-1,-1.73);
\draw[qqwuqq][line width=1pt][dotted][opacity=0.5] (1.41*0.5,-1.41*0.5) -- (1.41,-1.41);
\draw[qqwuqq][line width=1pt] (1.41*0.5,-1.41*0.5) -- (1.73,-1);
\draw[qqwuqq][line width=1pt] (1.41*0.5,-1.41*0.5) -- (1,-1.73);
\draw[orange][line width=2.5pt] (0,0) circle(1.5cm);
\draw (0,-2.1) node[below]{$\bfS_A$};
\draw (0,-2.3)[opacity=0] node[below]{$x$};
\end{tikzpicture}
\ \
\begin{tikzpicture} [scale=1.5]
\draw    [line width=1.5pt] (0,0) circle (2cm)[dotted];
\filldraw[color = black!25] (0,0) circle (1cm);
\draw    [line width=1.5pt] (0,0) circle (1cm)[dotted];
\foreach \x in{0,90,180,270}
\draw [white][line width=12pt][rotate = \x](0,0) -- (1.2,0);
\fill[blue!50] ( 2.  , 0.  ) circle (\pointsizeII);
\fill[blue] ( 1.  , 1.73) circle (\pointsizeII);
\fill[blue] ( 1.73, 1.  ) circle (\pointsizeII);
\fill[blue] ( 0.  , 2.  ) circle (\pointsizeII);
\fill[blue] (-1.  , 1.73) circle (\pointsizeII);
\fill[blue] (-1.73, 1.  ) circle (\pointsizeII);
\fill[blue] (-2.  , 0.  ) circle (\pointsizeII);
\fill[blue] (-1.73,-1.  ) circle (\pointsizeII);
\fill[blue] (-1.  ,-1.73) circle (\pointsizeII);
\fill[blue] ( 0.  ,-2.  ) circle (\pointsizeII);
\fill[blue] ( 1.  ,-1.73) circle (\pointsizeII);
\fill[blue] ( 1.73,-1.  ) circle (\pointsizeII);
\fill[blue] (-1.41*0.5,1.41*0.5) circle (\pointsizeII);
\fill[blue] (-1.41*0.5,-1.41*0.5) circle (\pointsizeII);
\fill[blue] (1.41*0.5,-1.41*0.5) circle (\pointsizeII);
\fill[blue] (1.41*0.5,1.41*0.5) circle (\pointsizeII);
\draw[blue][line width=1pt][dotted][opacity=0.5] (0,2)--(0,1);
\draw[blue][line width=1pt] (0,2)--(-1.41*0.5,1.41*0.5);
\draw[blue][line width=1pt] (0,2)--(1.41*0.5,1.41*0.5);
\draw[blue][line width=1pt][dotted][opacity=0.5] (-2,0)--(-1,0);
\draw[blue][line width=1pt] (-2,0)--(-1.41*0.5,1.41*0.5);
\draw[blue][line width=1pt] (-2,0)--(-1.41*0.5,-1.41*0.5);
\draw[blue][line width=1pt][dotted][opacity=0.5] (0,-2)--(0,-1);
\draw[blue][line width=1pt] (0,-2)--(-1.41*0.5,-1.41*0.5);
\draw[blue][line width=1pt] (0,-2)--(1.41*0.5,-1.41*0.5);
\draw[blue][line width=1pt][dotted][opacity=0.5] (2,0)--(1,0);
\draw[blue][line width=1pt][dotted][opacity=0.5] (2,0)--(1.41*0.5,1.41*0.5);
\draw[blue][line width=1pt][dotted][opacity=0.5] (2,0)--(1.41*0.5,-1.41*0.5);
\draw[blue][line width=1pt][dotted][opacity=0.5] (1.41*0.5,1.41*0.5) -- (1.41,1.41);
\draw[blue][line width=1pt] (1.41*0.5,1.41*0.5) -- (1.73,1);
\draw[blue][line width=1pt] (1.41*0.5,1.41*0.5) -- (1,1.73);
\draw[blue][line width=1pt][dotted][opacity=0.5] (-1.41*0.5,1.41*0.5) -- (-1.41,1.41);
\draw[blue][line width=1pt] (-1.41*0.5,1.41*0.5) -- (-1.73,1);
\draw[blue][line width=1pt] (-1.41*0.5,1.41*0.5) -- (-1,1.73);
\draw[blue][line width=1pt][dotted][opacity=0.5] (-1.41*0.5,-1.41*0.5) -- (-1.41,-1.41);
\draw[blue][line width=1pt] (-1.41*0.5,-1.41*0.5) -- (-1.73,-1);
\draw[blue][line width=1pt] (-1.41*0.5,-1.41*0.5) -- (-1,-1.73);
\draw[blue][line width=1pt][dotted][opacity=0.5] (1.41*0.5,-1.41*0.5) -- (1.41,-1.41);
\draw[blue][line width=1pt] (1.41*0.5,-1.41*0.5) -- (1.73,-1);
\draw[blue][line width=1pt] (1.41*0.5,-1.41*0.5) -- (1,-1.73);
%
\foreach \x in {0,90,180,270}
\fill[blue][rotate= \x] (-0.34, 1.97) circle (\pointsizeII);
\foreach \x in {0,90,180,270}
\draw[blue][rotate= \x] (-0.34, 1.97) to[out=-80,in=-90] (0,2) [line width=1pt][dashed];
\foreach \x in {0,90,180,270}
\fill[blue][rotate= \x] (-0.68, 1.87) circle (\pointsizeII);
\foreach \x in {0,90,180,270}
\draw[blue][rotate= \x] (-0.68, 1.87) -- (-0.71, 0.71) [line width=1pt][dashed];
\foreach \x in {0,90,180,270}
\draw[blue][rotate= \x] ( 0.26, 0.97) -- ( 0   , 2   ) [line width=1pt][dashed];
\foreach \x in {0,90,180,270}
\fill[blue][rotate= \x] ( 0.26, 0.97) circle (\pointsizeII);
\foreach \x in {0,90,180,270}
\draw[blue][rotate= \x] (-0.26, 0.97) -- ( 0   , 2   ) [line width=1pt][dotted][opacity=0.5];
\foreach \x in {0,90,180,270}
\draw[blue][rotate= \x] (-1.22, 1.58) -- (-0.71, 0.71) [line width=1pt][dotted][opacity=0.5];
\foreach \x in {0,90,180,270}
\fill[blue][rotate= \x] (-1.58, 1.22) circle (\pointsizeII);
\foreach \x in {0,90,180,270}
\draw[blue][rotate= \x] (-1.58, 1.22) -- (-0.71, 0.71) [line width=1pt][dashed];
\foreach \x in {0,90,180,270}
\fill[blue][rotate= \x] (-1.93, 0.52)  circle (\pointsizeII);
\foreach \x in {0,90,180,270}
\draw[blue][rotate= \x] (-1.93, 0.52) -- (-0.71, 0.71) [line width=1pt][dashed];
\draw[orange][line width=2.5pt] (0,0) circle(1.5cm);
\draw (0,-2.1) node[below]{$\bfS_{\BD(A)} \simeq \BD(\bfS_A)$};
\draw (0,-2.3)[opacity=0] node[below]{$x$};
\end{tikzpicture}
\caption{A permitted curve corresponding to a band (the $\rbullet$-FFAS $\rAS$ and $\rbullet$-marked points is ignored in this pictur for brevity)}
\label{fig:curve-band}
\end{figure}
Furthermore, by Proposition \ref{prop:BDsurface} (1), the marked surface $\bfS_{\BD(A)}$ of $\BD(A)$ is homotopic to $\BD(\bfS_A)$, it has a subsurface which is of the form shown in the second picture of \Pic \ref{fig:curve-band}.
It follows that the permissible curve without endpoint exists.
Thus, $\BD(A)$ is representation-infinite. Now, it is clear that $A$ is representation-infinite if and only if $\BD(A)$ is representation-infinite.
\end{remark}

\begin{proposition} \label{prop:repr type}
Let $A$ be a gentle algebra, then the following statements are equivalent:
\begin{itemize}
  \item[\rm(1)] $A$ is representation-finite;
  \item[\rm(2)] $\BD(A)$ is representation-finite;
  \item[\rm(3)] $A^{\CMA}$ is representation-finite.
\end{itemize}
\end{proposition}

\begin{proof}
Chen and Lu have proved that the representation types of $A$ and $A^{\CMA}$ coincide,
see \cite[Theorem 4.4]{CL2019}. On the other hand,
By Theorem \ref{thm:BR}, it is well-known that a gentle algebra is representation-infinite
if and only if its bound quiver has at least one band.
Then (1) and (2) are equivalent by Lemma \ref{lemm:BDband}.
\end{proof}

\begin{theorem} \label{mainthm:repr-type}
Let $\mathfrak{G} = \{ [A] \mid A=\kk\Q/\I \text{ is gentle } \}$ be a set of all isoclasses of gentle algebras,
where $[A]$ is the isoclass containing the gentle algebra $A$. Define
\[ \BD: \mathfrak{G} \to \mathfrak{G}, [A] \mapsto [\BD(A)]
\text{ and } (-)^{\CMA} : \mathfrak{G} \to \mathfrak{G}, [A] \mapsto [A^{\CMA}].  \]
Then $A$ is representation-finite if and only if for any finite sequence
$f_1, f_2, \cdots, f_n \in \{\BD, (-)^{\CMA}\}$,
$f_1f_2\cdots f_n(A)$ is representation-finite.
\end{theorem}

\begin{proof}
By Proposition \ref{prop:repr type}, $A$ is representation-finite if and only if $f_1(A)$ is representation-finite.
By Theorems \ref{thm:CL} and \ref{thm:BD(gent)}, $f_1(A)$ is also a gentle algebra.
Thus, by Proposition \ref{prop:repr type}, $f_1(A)$ is representation-finite if and only if $f_2f_1(A)$ is representation-finite.
One can show that this theorem holds by induction.
\end{proof}

\subsection{Derived representation types of gentle algebras} \label{subsect:D repr type}

Now we recall the definitions of homotopy string and homotopy band which are used to describe
the derived category $\Dcat^b(A)$ of a gentle algebra $A$ in \cite{ALP2016}.

\begin{definition} \rm \label{def:hstr hband}
A {\defines homotopy string} is a sequence $\hstr = p_1p_2\cdots p_m$ such that the following conditions hold.
\begin{itemize}
  \item[(1)] Each $p_i$ ($1 \=< i\=< m$) is either a nonzero path or the formal inverse of a non-zero path.
  \item[(2)] For $p_i=x_{i,1}\cdots x_{i,l_i}$ and $p_{i+1}=x_{i+1,1}\cdots x_{i+1,l_i}$
    ($x_{i,1}$, $\ldots$, $x_{i,2}$, $x_{i+1,1}$, $\ldots$, $x_{i+1,l_i}$ $\in \Q_1\cup\Q_1^{-1}$), we have:
  \begin{itemize}
    \item $\target(p_i) = \source(p_{i+1})$;
    \item if $p_i$ and $p_{i+1}$ are paths, then $x_{i,l_i}x_{i+1,1}\in\I$;
    \item if $p_i$ and $p_{i+1}$ are formal inverse of paths, then $x_{i+1,1}^{-1}x_{i,l_i}^{-1}\in\I$.
  \end{itemize}
  \item[(3)] If $p_i$ is a path and $p_{i+1}$ is a formal inverse of a path,
    then $x_{i,l_i} \ne x_{i+1,1}^{-1}$;
    if $p_i$ is a formal inverse of a path and $p_{i+1}$ is a path,
    then $x_{i,l_i}^{-1} \ne x_{i+1,1}$.
\end{itemize}
Each $p_i$ is called a {\defines homotopy letter}. Furthermore,
if $p_i$ is a path, then it is called a {\defines directed homotopy letter},
if $p_i$ is a formal inverse of a path, then it is called an {\defines inverse homotopy letter}
Two homotopy strings $\hstr$ and $\hstr'$ are said to be {\defines equivalent},
write it as $\hstr \simeq \hstr'$, if one of $\hstr = \hstr$ and $\hstr = (\hstr)^{-1}$ holds.

A {\defines homotopy band} is a homotopy string $\hband = p_1p_2\cdots p_m$ with $\target(p_m)=\source(p_1)$ satisfying the following conditions.
\begin{itemize}
  \item[(1)] It is not a non-trivial power of any homotopy string.
  \item[(2)] Its square $\hband^2$ is a homotopy string.
  \item[(3)] It has an equal number of direct and inverse homotopy letters $p_i$.
\end{itemize}
\end{definition}

For a gentle algebra $A=\kk\Q/\I$ and its BD-gentle algebra $\BD(A)=\kk\Q^{\BD}/\I^{\BD}$,
the following lemma shows that the existence of homotopy band on $(\Q,\I)$ and that of homotopy band on $(\Q^{\BD},\I^{\BD})$ coincide.

\begin{lemma} \label{lemm:BDhband}
The bound quiver $(\Q,\I)$ of a gentle algebra $A=\kk\Q/\I$ has a homotopy band
if and only if the bound quiver $(\Q^{\BD},\I^{\BD})$ of $\BD(A)$ has a homotopy band.
\end{lemma}

\begin{proof}
If $(\Q,\I)$ has a homotopy band $\hband$, assume that it is of the form
\begin{align}\label{formula:hband}
  \hband = p_1p_2\cdots p_m
\end{align}
\begin{center}
  (all $p_i$ are the homotopy letters given by $v_{i,1}\To{}{} v_{i,2}\To{}{} \cdots \To{}{} v_{1,l_i}$),
\end{center}
then each $p_i$ can be seen as a permitted path which is a subpath of some permitted thread $H_i$.
By Theorem \ref{thm:BD(gent)}, $p_i$ corresponding to the homotopy letter $q_i$ given by the following permitted path
\begin{align}\label{formula:BDhband}
  v_{i,1}' \To{a_{i,1}'}{} v_{i,1} \To{a_{i,1}''}{}
  v_{i,2}' \To{a_{i,2}'}{} v_{i,2} \To{a_{i,1}''}{} \cdots \To{}{}
  v_{1,l_i}' \To{a_{i,l_i}'}{} v_{1,l_i}.
\end{align}
The path above is a subpath of the permitted thread $H_i^{\BD}$ on $(\Q^{\BD},\I^{\BD})$.
Then we obtain that
\[ \hband^{\BD} = q_1q_2\cdots q_m \]
\begin{center}
  (all $q_i$ are the homotopy letters given by $p_i$)
\end{center}
is a homotopy band on $(\Q^{\BD},\I^{\BD})$ as required.

Conversely, if $(\Q^{\BD},\I^{\BD})$ has a homotopy band,
then it must be of the form given by (\ref{formula:BDhband}).
Using Theorem \ref{thm:BD(gent)} (or Lemma \ref{lemm:BD(Am)}),
we obtain a homotopy band on $(\Q,\I)$ which is given by
removing all vertices $v_{i,1}'$, $v_{i,2}'$, $\ldots$, $v_{i,l_i}'$
and all arrows $a_{i,1}'$, $a_{i,2}'$, $\ldots$, $a_{i,l_i}'$ ($1 \=< i \=< m$),
respectively reset targets of arrows $a_{i,1}''$, $\ldots$, $a_{i,l_i-1}''$
to $v_{i,2}$, $\ldots$, $v_{i,l_i}$.
It is of the form (\ref{formula:hband}).
\end{proof}

\begin{lemma} \label{lemm:CMAhband}
The bound quiver $(\Q,\I)$ of a gentle algebra $A=\kk\Q/\I$ has a homotopy band
if and only if the bound quiver $(\Q^{\CMA},\I^{\CMA})$ of $A^{\CMA}$ has a homotopy band.
\end{lemma}

\begin{proof}
For each homotopy band $\hband=p_1p_2\cdots p_m$ on $(\Q,\I)$,
all arrows on each $p_i = x_{i,1}x_{i,2}\cdots x_{i,l_i}$ ($1\=< i\=< m$) can be divided to two classes:

(1) arrows on some forbidden cycle;

(2) arrows that do not on any forbidden cycle.

Assume that $p_i$ is directed. By Theorem \ref{thm:CL}, if $x_{i,j}$ ($1\=< i\=< l_i$),
as an arrow on the quiver $\Q$, is an arrow on some forbidden cycle of $(\Q,\I)$,
then it changes to the composition $x_{i,j}^{-}x_{i,j}^{+}$ which is a non-zeor path of length two on $(\Q^{\CMA},\I^{\CMA})$;
otherwise, $x_{i,j}$ can be seen as an arrow on $\Q^{\CMA}$ naturally.
Thus, each homotopy letter $p_i$, as a nonzero path on $(\Q,\I)$, of $\hband$ corresponds to the following path on $(\Q^{\CMA},\I^{\CMA})$.
\[
q_i = y_{i,1}y_{i,2}\cdots y_{i,l_i}, \text{ where }
y_{i,j}
= \begin{cases}
x_{i,j}, \text{ if $x_{i,j}$ is an arrow on some forbidden cycle;} \\
x_{i,j}^{-}x_{i,j}^{+},  \text{ otherwise. }
\end{cases}
\]
The case for $p_i$ to be inverse is similar, it also corresponds to a formal inverse $q_i$ of a path on $(\Q^{\CMA},\I^{\CMA})$.
Then we obtain a sequence $\hband^{\CMA}=q_1q_1\cdots q_m$ which is a homotopy band on $(\Q^{\CMA},\I^{\CMA})$
(all $q_i$ are homotopy letters of $\hband^{\CMA}$) as required.

Conversely, if $(\Q^{\CMA},\I^{\CMA})$ has a homotopy band $\hband^{\CMA} = q_1q_2\cdots q_m$,
then for each homotopy letter $q_i$, if there is an arrow $\alpha^{+}$ (resp. $\alpha^{-}$) on $q_i$ which is obtained by an arrow $\alpha$ on $\Q$,
then we obtain that $\alpha^{-}$ (resp. $\alpha^{+}$) must be an arrow on $q_i$ by Theorem \ref{thm:CL}.
Thus, if $q_i$ is directed (resp. inverse), then it (resp. its formal inverse) is of the form
\[ \Path_{i} = y_{i,1}y_{i,2}\cdots y_{i,l_i}, \]
where each $y_{i,j}$ ($1\=< j\=< l_i$) is either an arrow $x_{i,j}$ which can be seen as an arrow on $\Q$
or a composition $x_{i,j}^{-}x_{i,j}^{+}$ of length two.
Notice that each $\alpha^{-}\alpha^{+}$ corresponds to an arrow $\alpha$ on $\Q$,
then the path $\Path_{i}$ corresponds to
\[ x_{i,1}x_{i,2}\cdots x_{i,l_i} :=
\begin{cases}
 p_i, & \text{if } q_i \text{ is directed; } \\
 p_i^{-1}, & \text{if } q_i \text{ is inverse }
\end{cases} \]
which is a path on $(\Q,\I)$.
Thus, we obtain a homotopy string $\hband = p_1p_2\cdots p_m$ which is a homotopy band on $(\Q,\I)$.
\end{proof}

\begin{remark} \label{rmk:BDhband CMAhband} \rm
We provide another proof of Lemmata \ref{lemm:BDhband} and \ref{lemm:CMAhband} by using marked surface in this remark.
Similar to Remark \ref{rmk:BDband}, we recall the definition of admissible curves.
Notice that if the marked surface $\bfS_A$ of a gentle algebra $A$ is used to describe the derived category $\Dcat^b(A)$,
we need assign a grading $\foliation$, say {\defines forliation}, of $\bfS_A$
which is a section of a projectivized tangent bundle $\mathbb{P}(T\Surf)$ of $\Surf$, cf. \cite{HKK2017, QZZ2022}.
Then the geometric model of $\Dcat^b(A)$ is defined as some {\defines graded marked surface}
$\bfS_A^{\foliation} = (\Surf_A, \Marked_A, \widetilde{\rAS}, {\foliation})$,
where $\widetilde{\rAS} = \{ \widetilde{\redanoind} \mid \redanoind \in \rAS\}$.
Here, for each smooth curve $c: [0,1] \to \Surf$, $\tc$ is a homotopy class of paths in the tangent space $\mathbb{P}(T_{c(t)}\Surf)$ of $\Surf$ at $c(t)$ from the subspace given by $\foliation$ to the tangent space of the curve, varying continuously with $t\in [0,1]$. We call $\tc$ is a {\defines grading} of $c$.
An {\defines admissible curve} is a pair $(c,\tc)$ which is composed of two data:
a curve $c$ and its grading $\tc$.
Here, $c$ is a sequence
\[ c = (c_{[i,i+1]})_{0 \=< i \=< n(c)} = c_{[0,1]} c_{[1,2]} \cdots c_{[n(c), n(c)+1]} \ (n(c)\in\NN) \]
of some segments given by one of the pictures given in \Pic \ref{fig:arc segment II}
\begin{figure}[htbp]
\centering
\begin{tikzpicture}[scale=1.25]
\draw[black] (-0.5*1.5, -0.86*1.5) -- ( 0.5*1.5, -0.86*1.5) [line width=1pt];
\draw[red] ( 0.5*1.5,-0.86*1.5) -- ( 1*1.5, 0) [line width=1pt];
\draw[red] ( 1*1.5, 0) -- ( 0.5*1.5, 0.86*1.5) [line width=1pt][dotted];
\draw[red] (-0.5*1.5, 0.86*1.5) -- ( 0.5*1.5, 0.86*1.5) [line width=1pt];
\draw[red] (-1*1.5, 0) -- (-0.5*1.5, 0.86*1.5) [line width=1pt][dotted];
\draw[red] (-0.5*1.5,-0.86*1.5) -- (-1*1.5, 0) [line width=1pt];
\fill[red] ( 0.5*1.5,-0.86*1.5) circle (0.1cm);
\fill[white] ( 0.5*1.5, -0.86*1.5) circle (1.8pt);
\fill[red] ( 1*1.5, 0) circle (0.1cm);
\fill[white] ( 1*1.5, 0) circle (1.8pt);
\fill[red] ( 0.5*1.5, 0.86*1.5) circle (0.1cm);
\fill[white] ( 0.5*1.5, 0.86*1.5) circle (1.8pt);
\fill[red] (-0.5*1.5, -0.86*1.5) circle (0.1cm);
\fill[white] (-0.5*1.5, -0.86*1.5) circle (1.8pt);
\fill[red] (-1*1.5, 0) circle (0.1cm);
\fill[white] (-1*1.5, 0) circle (1.8pt);
\fill[red] (-0.5*1.5, 0.86*1.5) circle (0.1cm);
\fill[white] (-0.5*1.5, 0.86*1.5) circle (1.8pt);
\fill[blue] (0,-0.86*1.5) circle (0.1cm) [line width=1pt];
\draw[violet] (0,-0.86*1.5) -- (0, 0.86*1.5) [line width=1pt];
\draw (0,-1.7) node{Case (1)};
\end{tikzpicture}
\ \
\begin{tikzpicture}[scale=1.25]
\draw[red] ( 0.5*1.5,-0.86*1.5) -- ( 1*1.5, 0) [line width=1pt][dotted];
\draw[red] ( 1*1.5, 0) -- ( 0.5*1.5, 0.86*1.5) [line width=1pt];
\draw[red] (-0.5*1.5, 0.86*1.5) -- ( 0.5*1.5, 0.86*1.5) [line width=1pt][dotted];
\draw[red] (-1*1.5, 0) -- (-0.5*1.5, 0.86*1.5) [line width=1pt];
\draw[red] (-0.5*1.5,-0.86*1.5) -- (-1*1.5, 0) [line width=1pt][dotted];
\fill[red] ( 0.5*1.5,-0.86*1.5) circle (0.1cm);
\fill[white] ( 0.5*1.5, -0.86*1.5) circle (1.8pt);
\fill[red] ( 1*1.5, 0) circle (0.1cm);
\fill[white] ( 1*1.5, 0) circle (1.8pt);
\fill[red] ( 0.5*1.5, 0.86*1.5) circle (0.1cm);
\fill[white] ( 0.5*1.5, 0.86*1.5) circle (1.8pt);
\fill[red] (-0.5*1.5, -0.86*1.5) circle (0.1cm);
\fill[white] (-0.5*1.5, -0.86*1.5) circle (1.8pt);
\fill[red] (-1*1.5, 0) circle (0.1cm);
\fill[white] (-1*1.5, 0) circle (1.8pt);
\fill[red] (-0.5*1.5, 0.86*1.5) circle (0.1cm);
\fill[white] (-0.5*1.5, 0.86*1.5) circle (1.8pt);
\draw[violet] (-1.15, 0.86*0.75) -- ( 1.15, 0.86*0.75) [line width=1pt];
\draw (0,-1.7) node{Case (2)};
\end{tikzpicture}
\caption{$\rAS$-arc segments}
\label{fig:arc segment II}
\end{figure}
(these segments are called {\defines $\rAS$-arc segments}) such that:
\begin{itemize}
  \item[(1)] $c(t_i) = c_{[i,i+1]}(t_i)$ and $c(t_{n(c)+1})=c_{[n(c),n(c)+1]}(t_{n(c)+1})$,
    where $0=t_0 < t_1 < \cdots < t_{n(c)} < t_{n(c)+1}=1$
    (the curve $c$ and each $\rAS$-arc segment $c_{[i,i+1]}$ are seen as function
    $c:[0,1] \to \Surf$ and $c_{[i,i+1]}: [t_i, t_{i+1}] \to \Surf$, respectively);
  \item[(2)] two adjacent $\bAS$-arc segment $c_{[i,i+1]}$ and $c_{[i+1,i+2]}$ are different,
    i.e., $c_{[i,i+1]}\ne c_{[i+1,i+2]}$ and $c_{[i,i+1]}\ne c_{[i+1,i+2]}^{-1}$;
  \item[(3)] $c(0), c(1) \in \Marked$ are $\bbullet$-marked points,
    or $c(0)=c(1)$ lies in $\Surf^{\circ}$;
  \item[(4)] $c$ is smooth, and if $c(0)=c(1)$, then $\tc(0)=\tc(1)$.
\end{itemize}

By \cite{QZZ2022}, each homotopy string $\hstr = p_1p_2\cdots p_{n-1}$ ($n\>= 1$)
corresponds to an admissible curve $\tc$ with $c = c_{[0,1]}c_{[1,2]}c_{[2,3]}\cdots c_{[n(c),n(c)+1]}$
such that each homotopy letter $p_i$ of $\hstr$ corresponds to the $\rAS$-arc segment $c_{[i,i+1]}$.
Under the correspondence as above, if $p_i$ is directed (resp. inverse),
then $c_{[i,i+1]}$ lies in some $\rAS$-elementary polygon $\redP$
such that $\marked(\redP)$, the $\bbullet$-marked point in the side $\partial\redP$, is left (resp. right) to $c_{[i,i+1]}$.
We call that $c_{[i,i+1]}$ is {\defines relatively directed} (resp. {\defines relatively inverse})
In particular, if $\hstr$ is a homotopy band, then $\tc$ is an admissible curve without endpoint.
In this case, the number of relatively directed $\rAS$-arc segments of $\tc$ equals that of relatively inverse $\rAS$-arc segments of $\tc$.

Next, we prove Lemma \ref{lemm:BDhband}.
Clearly, if $A$ is derived-discrete if and only if there is an admissible curve without endpoint as above in $\bfS_A^{\foliation}$, see the thickest curve in the first picture of \Pic \ref{fig}.
\begin{figure}[htbp]
\centering
\hspace{4mm}
\begin{tikzpicture}[scale=1.5]
\draw[line width = 1.5pt][fill=black!25][dotted] (0,0) circle (1cm);
\draw[line width = 1.5pt](1,0) arc(0:135:1);
\draw[line width = 1pt][red]
  (-3.00, 0.00) -- (-1.00, 0.00) arc(90:315:0.40);
\draw[line width = 1pt][red][rotate=45][dotted]
  (-1.00, 0.00) arc(90:315:0.40);
\draw[line width = 1pt][red][rotate=90][dotted]
  (-3.00, 0.00) -- (-1.00, 0.00) arc(90:315:0.40);
\draw[line width = 1pt][red][rotate=135]
  (-1.00, 0.00) arc(90:315:0.40);
\draw[line width = 1pt][ red ]
  ( 1.00, 0.00) to[out=  80,in= -90] ( 2.50, 1.50)
  ( 1.50, 2.50) to[out= 180,in=  90] ( 0.00, 1.00) to[out=  90,in=   0] (-1.00, 2.00);
\draw[line width = 1pt][ red ][dotted] ( 2.50, 1.50) to[out=  90,in=   0] ( 1.50, 2.50);
\draw[line width = 1pt][ red ][dotted] (-3.00, 0.00) arc(180:210:3);
\draw[line width = 1pt][ red ][dotted] (-1.50,-2.60) arc(240:270:3);
\draw[line width = 1pt][ red ][dotted] ( 0.00,-3.00) arc(270:300:3);
\draw[line width = 1pt][ red ][dotted] ( 2.60,-1.50) arc(330:395:3);
\draw[line width = 1pt][ red ][dotted] ( 0.00, 1.00) -- ( 0.00, 3.00);
\draw[line width=1.5pt][black] (-2.60,-1.50) arc(210:240:3);
\draw[line width=1.5pt][black] ( 1.50,-2.60) arc(300:330:3);
\foreach \x in {0,30,60,90,120,150,213}
\draw[line width = 1pt][ red ][fill=white][rotate = \x] (-3.00, 0.00) circle(\pointsizeIII);
\foreach \x in {0,45,...,180,270}
\draw[line width = 1pt][ red ][fill=white][rotate = \x] (-1.00, 0.00) circle(\pointsizeIII);
\draw[line width = 1pt][ red ][fill=white] ( 1.50, 2.50)  circle(\pointsizeIII);
\foreach \x in {225,292.5}
\fill[blue][rotate = \x] (-1.00, 0.00) circle(\pointsizeIII);
\foreach \x in {45,135}
\fill[blue][rotate = \x] (-3.00, 0.00) circle(\pointsizeIII);
\draw[line width = 2pt][orange]
  (-2.00, 0.00) to[out= -90,in= 180] ( 0.00,-2.00) to[out=   0,in= -90]
  ( 2.00, 0.00) to[out=  90,in=   0] ( 0.00, 2.00) to[out= 180,in=  45]
  (-1.41, 1.41);
\draw[line width = 2pt][orange][dotted]
  (-1.41, 1.41) arc(135:180:2);
\draw[blue][line width =.8pt] (-2.12,-2.12) -- (-2.00, 0.00);
\draw[blue][line width =.8pt] (-2.12,-2.12) -- (-1.38,-0.55);
\draw[blue][line width =.8pt] (-2.12,-2.12) to[out=100,in=-30] (-2.90,-0.78) [dotted];
\draw[blue][line width =.8pt] (-2.12,-2.12) to[out=-10,in=120] (-0.78,-2.90) [dotted];
\draw[blue][line width =.8pt] (-2.12,-2.12) -- (-0.55,-1.38) [dotted];
\draw[blue][line width =.8pt] (-2.12,-2.12) -- (-0.00,-2.00) [dotted];
\draw[blue][line width =.8pt][rotate=90] (-2.12,-2.12) -- (-2.00, 0.00) [dotted];
\draw[blue][line width =.8pt][rotate=90] (-2.12,-2.12) -- (-1.38,-0.55) [dotted];
\draw[blue][line width =.8pt][rotate=90] (-2.12,-2.12) to[out=100,in=-30] (-2.90,-0.78) [dotted];
\draw[blue][line width =.8pt][rotate=90] (-2.12,-2.12) to[out=-10,in=120] (-0.78,-2.90) [dotted];
\draw[blue][line width =.8pt][rotate=90] (-2.12,-2.12) -- (-0.55,-1.38);
\draw[blue][line width =.8pt] ( 2.12,-2.12) to[out=  90,in=  15] ( 0.71, 0.71);
\draw[blue][line width =.8pt] ( 0.71, 0.71) to[out=  90,in= -15] ( 0.20, 1.71);
\draw[blue][line width =.8pt] ( 0.71, 0.71) -- ( 2.50, 2.50) [dotted];
\draw[blue][line width =.8pt] (-0.38, 0.92) to[out=180,in=225] (-1.00, 2.00);
\draw[line width=1.5pt][fill=black!25][opacity=0] (0,-3) circle(0.5cm);
\draw (0,-3.5) node[below]{$\bfS_A$};
\end{tikzpicture}
\\
\
\\
\begin{tikzpicture}[scale=1.5]
\draw[line width = 1.5pt][fill=black!25][dotted] (0,0) circle (1cm);
\draw[line width = 1.5pt](1,0) arc(0:135:1);
\draw[line width = 1pt][red] (-2.48, 0.00) to[out=0,in=180] (-0.88,0.41);
\draw[line width = 1pt][red] (-2.48, 0.00) arc(90:-180:0.45) [dotted];
\foreach \x in {-22.5,0,22.5,157.5,180}
\draw[line width = 1pt][red][rotate = \x]
  (-1.00, 0.00) arc(90:290:0.20);
\foreach \x in {45.0,67.5,90,112.5,135}
\draw[line width = 1pt][red][rotate = \x][dotted]
  (-1.00, 0.00) arc(90:290:0.20);
\draw[line width = 1pt][red][dotted] ( 0.00,-1.00) -- ( 0.00,-2.48) arc(45:270:0.35);
\draw[line width = 1pt][red][rotate=90] (-2.48, 0.00) arc(90:-180:0.45) [dotted];
\draw[line width = 1pt][red]
  ( 0.88, 0.39) to[out=  30,in= -90] ( 2.50, 1.50)
  ( 1.50, 2.50) to[out= 180,in=  90] ( 0.00, 1.00) to[out=  90,in=   0] (-1.00, 2.00);
\draw[line width = 1pt][red][dotted] ( 2.50, 1.50) to[out=  90,in=   0] ( 1.50, 2.50);
\draw[line width = 1pt][red] ( 0.88, 0.41) arc(-45:120:0.3);
\draw[line width = 1pt][red][rotate=195] (-3.00, 0.00) arc(180:200:3.0);
\draw[line width = 1pt][red] ( 0.00, 1.00) arc(0:207:0.28);
\draw[line width = 1pt][ red ][dotted] (-3.00, 0.00) arc(180:210:3);
\draw[line width = 1pt][ red ][dotted] (-1.50,-2.60) arc(240:270:3);
\draw[line width = 1pt][ red ][dotted] ( 0.00,-3.00) arc(270:300:3);
\draw[line width = 1pt][ red ][dotted] ( 2.60,-1.50) arc(330:395:3);
\draw[line width=1.5pt][black] (-2.60,-1.50) arc(210:240:3);
\draw[line width=1.5pt][black] ( 1.50,-2.60) arc(300:330:3);
\draw[line width = 1pt][ red ][dotted] ( 0.00, 1.00) -- ( 0.00, 3.00);
\draw[line width=1.5pt][dotted][fill=black!25] (0,-3) circle(0.5cm);
\draw[line width=1.5pt][dotted][fill=black!25] (-3,0) circle(0.5cm);
\foreach \x in {9.5,30,60,70.25,80.5,
                99.5,120,150,180,196,213}
\draw[line width = 1pt][ red ][fill=white][rotate = \x] (-3.00, 0.00) circle(\pointsizeIII);
\foreach \x in {-22.5,0,22.5,45,67.5,90,
                112.5,135,157.5,180,202.5,
                240,270,302}
\draw[line width = 1pt][ red ][fill=white][rotate = \x] (-1.00, 0.00) circle(\pointsizeIII);
\draw[line width = 1pt][ red ][fill=white]
  ( 1.50, 2.50) circle(\pointsizeIII)
  ( 0.00,-2.48) circle(\pointsizeIII)
  (-2.48, 0.00) circle(\pointsizeIII)
  ( 0.82, 2.32) circle(\pointsizeIII);
\foreach \x in {222,255,285,315}
\fill[blue][rotate = \x] (-1.00, 0.00) circle(\pointsizeIII);
\foreach \x in {45,135}
\fill[blue][rotate = \x] (-3.00, 0.00) circle(\pointsizeIII);
\draw[line width = 2pt][orange]
  (-2.00, 0.00) to[out= -90,in= 180] ( 0.00,-2.00) to[out=   0,in= -90]
  ( 2.00, 0.00) to[out=  90,in=   0] ( 0.00, 2.00) to[out= 180,in=  45]
  (-1.41, 1.41);
\draw[line width = 2pt][orange][dotted]
  (-1.41, 1.41) arc(135:180:2);
\draw[blue][line width =.8pt] (-2.12,-2.12) to[out= 100,in= -30] (-2.90,-0.95) [dotted];
\draw[blue][line width =.8pt] (-2.12,-2.12) to[out=  99,in= -50] (-2.45,-0.88) [dotted];
\draw[blue][line width =.8pt] (-2.12,-2.12) to[out=  90,in=-110] (-1.75, 0.19) [dashed];
\draw[blue][line width =.8pt] (-2.12,-2.12) to[out=  80,in= 210] (-1.18, 0.15);
\draw[blue][line width =.8pt] (-2.12,-2.12) to[out=  75,in= 210] (-1.18,-0.21) [dashed];
\draw[blue][line width =.8pt] (-2.12,-2.12) to[out=  70,in= 210] (-0.98,-0.68);
\draw[blue][line width =.8pt] (-2.12,-2.12) -- (-0.68,-0.98) [dotted];
\draw[blue][line width =.8pt] (-2.12,-2.12) -- (-0.21,-1.18) [dotted];
\draw[blue][line width =.8pt] (-2.12,-2.12) to[out=   0,in= 150] (-0.55,-2.53) [dotted];
\draw[blue][line width =.8pt] (-2.12,-2.12) to[out= -10,in= 110] (-0.75,-2.93) [dotted];
\draw[blue][line width =.8pt] (-2.12,-2.12) to[out= -10,in= 110] (-1.25,-2.73) [dotted];
\draw[blue][line width =.8pt] (-2.12,-2.12) to[out=  10,in= 180] (-0.00,-1.85) [dotted];
\draw[blue][line width =.8pt] (-0.00,-1.85) to[out=   0,in= 170] ( 2.12,-2.12) [dotted];
\draw[blue][line width =.8pt] ( 2.12,-2.12) -- ( 0.80,-2.15) [dotted];
\draw[blue][line width =.8pt][rotate=90] (-2.12,-2.12) to[out= 100,in= -30] (-2.90,-0.95) [dotted];
\draw[blue][line width =.8pt][rotate=90] (-2.12,-2.12) to[out=  70,in= 210] (-1.18,-0.21) [dotted];
\draw[blue][line width =.8pt][rotate=90] (-2.12,-2.12) to[out=  70,in= 210] (-0.98,-0.68) [dotted];
\draw[blue][line width =.8pt][rotate=90] (-2.12,-2.12) -- (-0.68,-0.98) [dotted];
\draw[blue][line width =.8pt][rotate=90] (-2.12,-2.12) -- (-0.21,-1.18);
\draw[blue][line width =.8pt] ( 2.12,-2.12) to[out= 100,in= -45] ( 1.20, 0.15) [dashed];
\draw[blue][line width =.8pt] ( 2.12,-2.12) to[out=  90,in= -90] ( 2.12, 1.02) arc(0:180:0.93);
\draw[blue][line width =.8pt] ( 2.12,-2.12) -- ( 2.68, 1.31) [dashed];
\draw[blue][line width =.8pt] ( 2.12,-2.12) -- ( 2.96, 0.51) [dotted];
\draw[blue][line width =.8pt] ( 2.12,-2.12) -- ( 2.96,-0.51) [dotted];
\draw[blue][line width =.8pt] ( 0.26, 0.98) arc(180:-70:0.35) [dashed];
\draw[blue][line width =.8pt] ( 0.26, 0.98) to[out= 120,in=-120] ( 0.26, 1.85);
\draw[blue][line width =.8pt] ( 0.26, 0.98) to[out= 100,in=-120] ( 0.46, 1.75)
                                            to[out=  60,in=-150] ( 1.25, 2.48) [dashed];
\draw[blue][line width =.8pt] (-0.26, 0.98) arc( 0:235:0.32) [dashed];
\draw[blue][line width =.8pt] (-0.71, 0.71) to[out= 180,in= 225] (-1.00, 2.00);
\draw (0,-3.5) node[below]{$\bfS_{\BD(A)} \simeq \BD(\bfS_A)$};
\end{tikzpicture}
\caption{An admissible curve corresponding to a homotopy band}
\label{fig}
\end{figure}
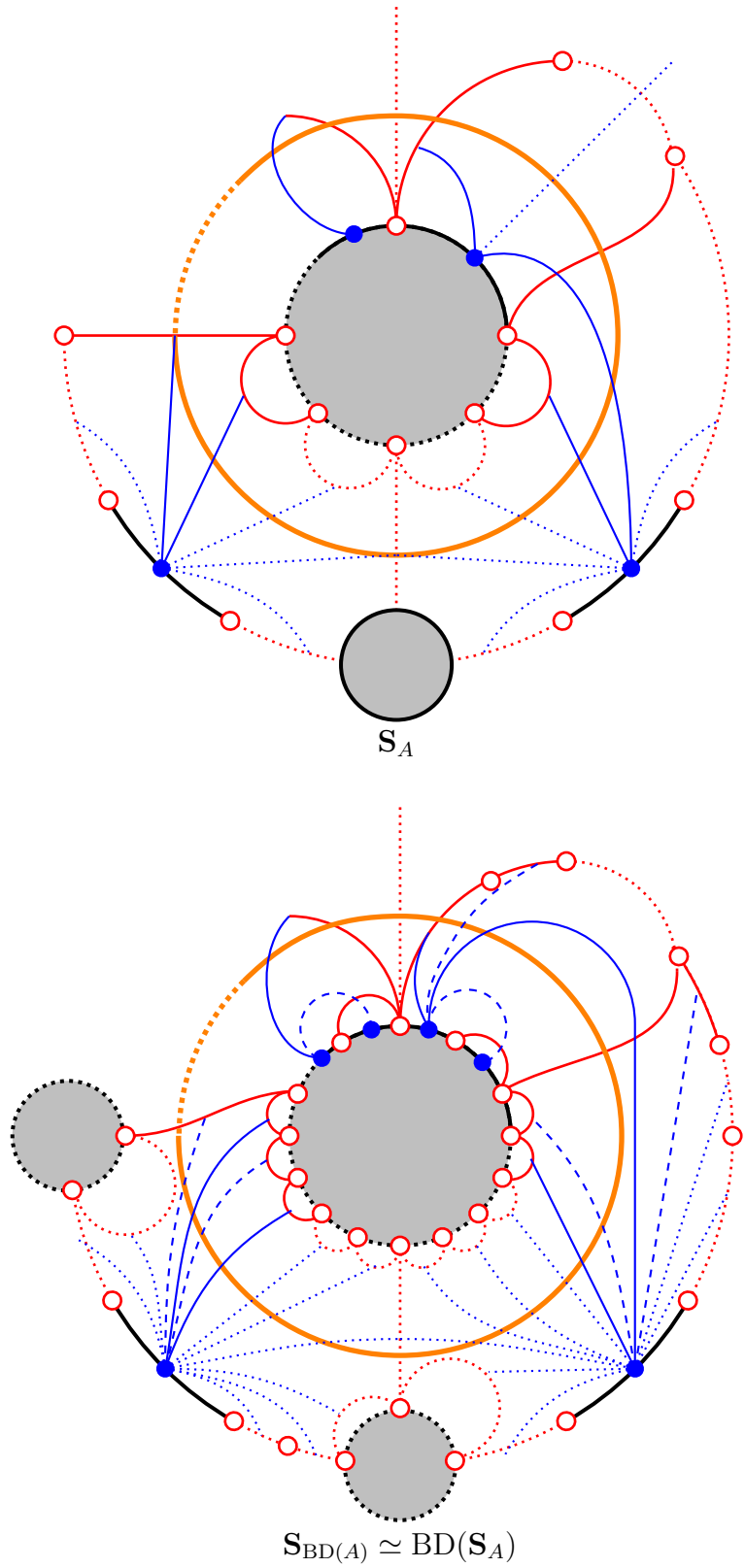
In this case, the permissible curve without endpoint corresponds to a homotopy band on the bound quiver of $A$.
By Proposition \ref{prop:BDsurface} (1), the marked surface $\bfS_{\BD(A)}$ of $\BD(A)$ is homotopic to $\BD(\bfS_A)$.
Consider the graded marked surface $\BD(\bfS_A)^{\foliation}$,
where the foliation $\foliation$ is also the foliation of $\bfS_A^{\foliation}$
since $\BD(\bfS_A)$ is obtained by, respectively, adding some marked points and arcs to $\Marked_A$ and $(\rAS)_A$.
It is easy to see that $\BD(\bfS_A)^{\foliation}$ has a subsurface $\Surf'$
(with the restricted foliation $\foliation|_{\Surf'}$)
which is of the form shown in the second picture of \Pic \ref{fig}.
It follows that the admissible curve without endpoint exists. Thus, $\BD(A)$ is derived-discrete.
Now, it is clear that $A$ is derived-discrete if and only if $\BD(A)$ is derived-discrete
since $\bfS_A^{\foliation}$ can be seen as a graded marked surface obtained by deleting
some marked points in $\Marked_{\BD(A)}$ and some arcs in $(\rAS)_{\BD(A)}$ from $\bfS_{\BD(A)}$.

The proof of Lemma \ref{lemm:CMAhband} need Proposition \ref{prop:BDsurface} (2),
and it is similar to that of Lemma \ref{lemm:BDhband}.
\end{remark}

The following result has been shown by Bekkert and Merklen.

\begin{theorem}[{\!\!\cite[Theorem 4 (b)]{BM2003}}] \label{thm:BM}
A gentle algebra $A$ is derived-discrete if and only if its bound quiver $(\Q,\I)$ does not have any homotopy band.
\end{theorem}

Thus, we obtain the following result.

\begin{proposition} \label{prop:der-disc}
Let $A$ be a gentle algebra. Then the following statements are equivalent:
\begin{itemize}
  \item[\rm(1)] $A$ is derived-discrete;
  \item[\rm(2)] $\BD(A)$ is derived-discrete;
  \item[\rm(3)] $A^{\CMA}$ is derived-discrete.
\end{itemize}
\end{proposition}

\begin{proof}
Since $A$ is derived-discrete, we obtain that the bound quiver $(\Q,\I)$ of $A$ does not have any homotopy band by Theorem \ref{thm:BM}.
It follows that the bound quiver $(\Q^{\BD},\I^{\BD})$ of $\BD(A)$ does not have any homotopy band by Lemma \ref{lemm:BDhband}.
Thus, $\BD(A)$ is discrete-derived.
Conversely, we can prove that $A$ is derived-discrete if $\BD(A)$ is derived-discrete by using Lemma \ref{lemm:BDhband} and Theorem \ref{thm:BM}.
Therefore, (1) and (2) are equivalent.

By Lemma \ref{lemm:CMAhband} and Theorem \ref{thm:BM},
one can check that (1) and (3) are equivalent in a similar way.
\end{proof}

\begin{sloppypar}
\begin{theorem} \label{mainthm:der repr-type}
Keep the notations from Theorem \ref{mainthm:repr-type}.
Then $A$ is representation-discrete if and only if for any finite sequence
$f_1, f_2, \cdots, f_n \in \{\BD, (-)^{\CMA}\}$,
$f_1f_2\cdots f_n(A)$ is representation-discrete.
\end{theorem}
\begin{proof}
By Proposition \ref{prop:der-disc}, $A$ is a representation-discrete if and only if $f_1(A)$ is representation-discrete.
By Theorems \ref{thm:CL} and \ref{thm:BD(gent)}, $f_1(A)$ is also a gentle algebra.
Thus, by Proposition \ref{prop:repr type}, $f_1(A)$ is representation-discrete if and only if $f_2f_1(A)$ is representation-discrete.
One can show that this theorem holds by induction.
\end{proof}
\end{sloppypar}

\section*{Acknowledgements}

The authors would like to thank Houjun Zhang for helpful discussions.

\section*{Funding}

Yu-Zhe Liu is supported by
National Natural Science Foundation of China (Grant Nos. 12401042, 12171207),
Guizhou Provincial Basic Research Program (Natural Science) (Grant No. ZK[2024]YiBan066)
and Scientific Research Foundation of Guizhou University (Grant Nos. [2023]16, [2022]53, [2022]65).

Chao Zhang is supported by
National Natural Science Foundation of China (Grant Nos. 12461006),
Guizhou Provincial Basic Research Program (Natural Science) (Grant No. ZD[2025]085).







\def\cprime{$'$}

\end{document}